\documentclass[reqno]{amsart}
\usepackage{graphicx} 
\usepackage{geometry}
\usepackage{color}
\usepackage{enumitem}
\usepackage{hyperref}
\usepackage{verbatim}
\usepackage{todonotes}

\title{Continuum of coupled Wasserstein gradient flows}

\author{Clément Cancès}
\address{Univ. Lille, CNRS, Inria, UMR8524 -- Laboratoire Paul Painlevé, F-59000 Lille}
\email{clement.cances@inria.fr}

\author{Daniel Matthes}
\address{Department of Mathematics, TUM School of Computation, Information and Technology, Technical University of Munich, Germany}
\email{matthes@ma.tum.de}

\author{Ismael Medina}
\address{Institute of Computer Science, University of Göttingen, Göttingen, Germany}
\email{ismael.medina@cs.uni-goettingen.de}

\author{Bernhard Schmitzer}
\address{Institute of Computer Science, University of Göttingen, Göttingen, Germany}
\email{schmitzer@cs.uni-goettingen.de}

\date{\today}

\newcommand{\R}{{\mathbb R}}

\newcommand{\Rnn}{{\mathbb R}_{\ge0}}
\newcommand{\dd}{\,{\mathrm d}}
\newcommand{\nml}{{\mathbf n}}
\newcommand{\eps}{\varepsilon}
\newcommand{\nrg}{{\mathcal E}}
\newcommand{\entr}{{\mathcal H}}
\newcommand{\adm}{{\Gamma(\mu, \nu)}}
\newcommand{\indic}{\iota}
\newcommand{\id}{\operatorname{id}}
\newcommand{\wass}{{\mathbf W}}
\newcommand{\wassF}{\wass_F}
\newcommand{\intX}{\int_X}

\newcommand{\intXY}{\int_{X\times Y}}

\newcommand{\weakto}{\rightharpoonup}
\newcommand{\proj}{\mathrm{P}}

\newcommand{\tn}[1]{\textrm{#1}}
\DeclareMathOperator*{\argmin}{argmin}
\newcommand{\meas}{\mathcal{M}}
\newcommand{\measp}{\mathcal{M}_+}
\newcommand{\prob}{\mathcal{P}}
\newcommand{\assign}{:=}
\newcommand{\veps}{\varepsilon}
\newcommand{\diam}{\tn{diam}}
\newcommand{\la}{\left\langle}
\newcommand{\ra}{\right\rangle}
\newcommand{\Lebesgue}{\mathcal{L}}
\newcommand{\DU}{\overline{D}_\veps}

\newcommand{\rhod}{r}
\newcommand{\varrhod}{\tilde{r}}

\newtheorem{lemma}{Lemma}
\theoremstyle{definition}
\newtheorem{remark}{Remark}

\begin{document}

\begin{abstract}
We study a system of drift-diffusion PDEs for a potentially infinite number of incompressible phases, subject to a joint pointwise volume constraint.
Our analysis is based on the interpretation as a collection of coupled Wasserstein gradient flows or, equivalently, as a gradient flow in the space of couplings under a `fibered' Wasserstein distance.
We prove existence of weak solutions, long-time asymptotics, and stability with respect to the mass distribution of the phases, including the discrete to continuous limit.
A key step is to establish convergence of the product of pressure gradient and density, jointly over the infinite number of phases. 
The underlying energy functional is the objective of entropy regularized optimal transport, which allows us to interpret the model as the relaxation of the classical Angenent--Haker--Tannenbaum (AHT) scheme to the entropic setting.
However, in contrast to the AHT scheme's lack of convergence guarantees, the relaxed scheme is unconditionally convergent.
We conclude with numerical illustrations of the main results. 
\end{abstract}

\keywords{multiphase porous media flow, fibered Wasserstein gradient flow, minimizing movement scheme, Angenent—Haker—Tannenbaum scheme, entropic optimal transport}

\subjclass[2020]{
  35K51, 35A15, 35A35
}

	\maketitle

	\section{Introduction}
	\label{sec:intro}
	\subsection{Problem setup}
	Let $X\subset \R^d$ be the closure of a bounded and connected open set with smooth boundary, and let $Y\subset \R^{d'}$ be compact. 
	Let $\Lebesgue$ denote the Lebesgue measure on $X$, $\mu$ be a $\Lebesgue$-absolutely continuous probability measure on $X$ and $\nu$ an arbitrary probability measure on $Y$.
	
	We will denote by $\rho$ a probability measure on $X\times Y$, absolutely continuous with respect to $\Lebesgue \otimes \nu$ and with density $\rhod$ in $L^1(X\times Y, \Lebesgue \otimes \nu)$. Furthermore, let  $V\in C^2(X\times Y)$ be a non-negative function and $\kappa>0$ be a positive parameter. 
	
	Motivated on the one hand by multiphase porous media flows ---see for instance \cite{BB90} for a presentation of the physical models, \cite{Chen01} for some elements of mathematical analysis in the two-phase context, and for \cite{otto1999evolution, CGM15, CGM17} for a variational reinterpretation of the multiphase flow problem (see also \cite{OE97, CMN19, CM23})---, and on the other hand by the classical Angenent-Haker-Tannenbaum (AHT) scheme for $L^2$ optimal transport \cite{AHT} --- see below for a detailed explanation --- we study the following stratified gradient flow:
	\begin{align}
	\label{eq:PDE1}
	\partial_t\rhod(t,x,y) &= \kappa\Delta_X\rhod(t,x,y) + \nabla_X\cdot\big(\rhod(t,x,y)\,\nabla_X[V(x,y)-\Pi(x)]\big), \\
	\label{eq:PDE1b}
	0 &= \nml\cdot\big(\kappa\nabla_X\rhod(t,x,y)+\rhod(t,x,y)\,\nabla_X[V(x,y)-
	\Pi(t,x)]\big) \quad \text{for $x\in\partial X$}, \\
	\label{eq:PDE2}
	\mu(x) &= \int_Y\rhod(t,x,y)\dd \nu(y), 
	\end{align}
    where $\nabla_X$ and $\nabla_X\cdot$ denote the gradient and the divergence with respect to the variable $x \in X$ only. 

	The pressure $\Pi$ in \eqref{eq:PDE1} is implicitly given via the \textit{first marginal} constraint \eqref{eq:PDE2}, 
	where by slight abuse the notation we use the symbol $\mu(x)$ for the (Lebesgue) density of the Lebesgue-absolutely continuous measure $\mu$.
	Note that the mass at level $y\in Y$, that we denote \textit{second marginal}, is constant in time thanks to the divergence structure of \eqref{eq:PDE1} and the no-flux boundary condition \eqref{eq:PDE1b}. 
	Assuming the initial condition has second marginal $\nu$, it will be subsequently preserved over time.
 	
	Formally, \eqref{eq:PDE1}--\eqref{eq:PDE2} is a $y$-wise Wasserstein gradient flow for the potential energy
	\begin{align}
	\label{eq:energy}
	\nrg(\rho) = \int_{X\times Y} V(x,y) \dd \rho(x,y) + \kappa\entr(\rho) +\indic_{\adm}(\rho),
	\end{align}
	where $\entr$ denotes the weighted negative entropy 
	\begin{equation}
	\label{eq:entr}
	\entr(\rho\mid \nu) := \int_{X\times Y} (\rhod \log \rhod -\rhod + 1) \dd x \dd \nu(y)
	\end{equation}
	(we will write $\entr(\rho)$ whenever $\nu$ is clear from the context)
	and $\indic_\adm$ is the indicator function of the feasible couplings (or transport plans) between $\mu$ and $\nu$, see \eqref{eq:coupling} below for a rigorous definition. 
	
	\subsection{First motivation: multiphasic flows}
	\label{sec:motivation}
	\eqref{eq:PDE1}--\eqref{eq:PDE2} arises naturally as a generalization of multiphase models such as those studied in \cite{CGM17,CZ22} to the case of a (possibly) infinite number of species. 
 	Indeed,	Equations \eqref{eq:PDE1}--\eqref{eq:PDE2} boil down to such models when the space $Y$ is finite, or equivalently if the measure $\nu$ is made of a finite number of Dirac masses:
	 \begin{equation}
	 \nu = \sum_{i=1}^k m_i \delta_{y_i},
	 \qquad 
	 \text{with } (y_i)_{i=1}^k \subset Y
	 \text{ and } \sum_{i=1}^k m_i = 1.
 	\end{equation}
	
	The measure $\mu$ then corresponds to the porosity of the material, i.e.~the local ratio of void space available for the fluid in the  solid porous structure, within which the multiphase and incompressible fluid is flowing.
	The function $V(\cdot, y_i)$ then denotes the intrinsic potential energy of the $i$-th species.
	Condition~\eqref{eq:PDE2} is then interpreted as the constraint that the whole porous space is occupied by the fluid mixture.

	In our setting, $\kappa$ is a given positive constant that encodes the intensity of the term $\kappa\entr(\rho)$, which could be interpreted as a capillary energy term. The limiting case of vanishing capillarity $\kappa=0$ is known to be highly complex. Even in the seemingly simple two-phase case, the well-posedness of the problem remains widely open. In the one-dimensional case $d=1$, the problem reduces to some Burgers' type equation~\cite{otto1999evolution}. It is now well understood that in this particularly simple setting, the JKO scheme selects the unique entropy solution in Kruzkov's sense~\cite{otto1999evolution, GO13}. In larger dimension, tools from convex integration can even be used to construct multiple weak solutions \cite{CFG11, Szekelyhidi2012, FS18}. Moreover when the space dimension $d$ is greater or equal to two, there is no available criterion to charaterize the solution captured by the JKO scheme, and up to our knowledge there is also no general existence theorem unless the external potential is identically equal to zero~\cite{Vig96}. 
	
 	Let us expand on the physical interpretation of \eqref{eq:PDE1}-\eqref{eq:PDE2}.
	Formally, \eqref{eq:PDE1} can be written as: 
 	\begin{equation}
 	\label{eq:continuity-equation}
 		\partial_t \rhod_t + \nabla_X \cdot (r_t v_t) = 0,
 		\qquad \text{with}
 		\quad
 		v_t(x,y) \assign \nabla_X [\Pi_t(x) - V(x,y) - \kappa \log r_t(x,y)].
 	\end{equation}
 	
 	Integrating \eqref{eq:continuity-equation} with respect to the $Y$-marginal $\nu$ yields:
 	\begin{equation}
 		\label{eq:divergence-free}
 		0 = \partial_t \mu = \partial_t \left(
 		\int_Y r_t(x,y) \dd \nu(y)
 		\right)
 		= 
 		- \nabla_X \cdot \left(
 		\int_Y r_t(x,y) v_t(x,y) \dd \nu(y)
 		\right).
 	\end{equation}
 	Equation \eqref{eq:divergence-free} shows that the $X$-projection of the momentum field $\rho_t v_t$ results in an $X$-divergence-free velocity field, therefore preserving the pointwise $X$-marginal constraint. 
 	In fact, the following decomposition (closely related to the Helmholtz decomposition) holds: 
 	\begin{equation}
 		-\nabla_X [V(x,y) + \kappa \log \rhod_t(x,y)] 
 		= 
 		-\nabla \Pi_t(x) + v_t(x,y).
 	\end{equation}
 	
 	The left-hand side corresponds to the steepest descend direction for \eqref{eq:energy} once one removes the marginal constraint.
 	Thus, the velocity field $v_t(x,y)$ encodes the divergence-free part of the unconstrained velocity field, with $\nabla \Pi_t(x)$ claiming the gradient component. 
 	Whenever the left-hand side can be written as a sheer gradient of a function on $X$, the velocity field $v_t$ becomes identically zero, resulting in a stationary point of  \eqref{eq:PDE1}-\eqref{eq:PDE2}. 
  Intuitively, $v_t$ will progressively cancel the curl component of the left-hand side until such a stationary point is attained.
 	This will become clearer when we explore the connections to the AHT scheme \cite{AHT} in Section \ref{sec:motivation-aht}.
 	
 	For future reference, let us conclude by stating the elliptic equation satisfied by $\Pi_t$, which is a direct consequence of \eqref{eq:divergence-free} plus the marginal constraints:
 	\begin{equation}
 	\label{eq:PDE3}
 	\nabla \cdot [\mu(x) \nabla \Pi_t(x)] 
 	= 
 	\kappa\Delta \mu(x) 
 	+
 	\nabla \cdot \left[
 	\int_Y \nabla_X V(x,y) \rhod(t,x,y) \dd \nu(y)
 	\right].
 	\end{equation}

	\subsection{Second motivation: Angenent-Haker-Tannenbaum scheme}
	\label{sec:motivation-aht}
	The field of computational optimal transport offers an alternative interpretation for the evolution equations \eqref{eq:PDE1}--\eqref{eq:PDE2}. 
	From this perspective, the energy \eqref{eq:energy} is the entropic transport score of the coupling $\rho$ between $\mu$ and $\nu$ for the cost function $V$ and regularization strength $\kappa$. This entropic objective is a common approximation to the classical, unregularized optimal transport problem (corresponding to $\kappa = 0$), which has gained increased popularity thanks to the availability of efficient solvers, most notably the  Sinkhorn algorithm \cite{cuturi2013,PeyreCuturiCompOT} and the connection to stochastic optimal control problems \cite{Chen2016, TreeMultiMarginal2020, Lavenant2021}.
	From this perspective, the evolution equations \eqref{eq:PDE1}--\eqref{eq:PDE2} address the  solution of the entropic transport problem by evolving the density of the transport plan in the direction of steepest descent of the entropic score $\nrg$, conditioned to keeping the marginals feasible.

	A similar philosophy is central to the Angenent-Haker-Tannenbaum (AHT) scheme \cite{AHT,SantambrogioOT}, a classical algorithm for unregularized optimal transport (i.e.~$\eqref{eq:energy}$ for $\kappa = 0$) and $V$ strictly convex. In this regime the minimizers of $\nrg$ are very sparse ---when $\mu$ has a density, each $x\in X$ sends mass to a single $y\in Y$ \cite{MonotoneRerrangement-91, McCannGangboOTGeometry1996} 
	--- which motivates the parametrization of the current plan $\rho$ as a \textit{transport map}.
	Hence, the AHT scheme considers a minimizing flow for $\nrg$ in the space of transport maps, progressively removing the curl component of an initial transport map $T_0$ until only the gradient component remains. Formally, for $T_0: X\rightarrow Y$ verifying $\nu = T_{0\sharp} \mu$, the AHT evolution is driven by: 
	\begin{equation}
	\label{eq:AHTFlowGeneral}
	T_t = T_0 \circ s_t^{-1} \qquad \text{with} \qquad s_0=\text{id}_X, \qquad \partial_t s_t = v_t \circ s_t,
	\qquad \nabla \cdot (\mu v_t )=0
	\end{equation}
	where $v_t$ is a velocity field whose divergence constraint ensures $(s_{t})_\sharp \mu = \mu$, and therefore ---if such $s_t$ exists--- $T_t$ stays a feasible transport map between $\mu$ and $\nu$ at all times $t\ge 0$. $v_t$ is chosen among the feasible velocity fields to realize the steepest descent for the transport cost
	\begin{equation}
	\nonumber
	E(T) = \int_X V(x,T(x)) \mu(x) \dd x, \qquad 
	\partial_t E(T_t) = \int_X \langle \nabla_X V(x,T_t(x)), v_t(x) \rangle  \mu(x) \dd x 
	\end{equation}
	with respect to some metric of choice. 
	In \cite{AHT} the velocity field is chosen as $v_t(x) = u_t(x) / \mu(x)$, where $u_t(x)$ is the $L^2(X)$ projection of $x\mapsto -\nabla_X V(x, T_t(x))$ onto the space of divergence-free velocity fields.
	If one chooses instead the $L^2(X, \mu)$-metric, the steepest descent is given by setting $v_t$ to the projection of the vector field $x \mapsto -\nabla_X V(x,T_t(x))$ onto the subspace satisfying $\nabla \cdot (\mu v_t)=0$ in $L^2(X, \mu)$. 
	This projection is given by
	\begin{equation}
	\label{eq:projection-AHT}
	v_t(x) := \nabla \Pi(x) - \nabla_X V(x,T_t(x))
	,
	\quad
	\tn{with }
	\nabla \cdot [\mu(x) \nabla \Pi(x)] =
	\nabla \cdot 
	[\mu(x)\nabla_X V(x,T_t(x))]
	\end{equation}
	
	Leaving aside the (challenging) problem of existence of solutions to \eqref{eq:AHTFlowGeneral}-\eqref{eq:projection-AHT}, the similarities between  \eqref{eq:continuity-equation} and \eqref{eq:projection-AHT} are striking.
    Indeed, $\Pi$ satisfies both in \eqref{eq:PDE3} and in \eqref{eq:projection-AHT} the same elliptic equation, once we (formally) allow in the former the choice $\kappa = 0$ and that $\rho_t$ is given by a Monge plan $\rho_t = (\id, T_t(x))_\sharp \mu$. 
    Moreover, the AHT velocity field $v_t(x) = \nabla \Pi(x) - \nabla_X V(x,T_t(x))$ is also analogous to that in \eqref{eq:continuity-equation}, once we consider that in the AHT only points of the form $(x, T_t(x))$ carry mass, and therefore the velocity field can be given as a function of $x$ instead of $(x,y)$. For these reasons one may additionally interpret \eqref{eq:PDE1}--\eqref{eq:PDE2} as a relaxation of the AHT scheme to entropic optimal transport, which inevitably requires the relaxation from transport maps to transport plans.
    \smallskip
	
	Leaving aside the well-posedness issues of the AHT scheme --- which largely align with those outlined in Section \ref{sec:motivation} for the limit of vanishing $\kappa$ in  multiphase flows---, the main drawback of the AHT scheme is that it is unable to escape certain suboptimal configurations. Even though, by construction, the score $E(T)$ is non-increasing along solutions of \eqref{eq:AHTFlowGeneral}-\eqref{eq:projection-AHT}, the stationary points are all the feasible maps $T$ that verify that $x \mapsto \nabla_X V(x, T(x))$ is curl-free. Unfortunately this characterization does not suffice to grant optimality, and may in some cases include not only the best but also the \textit{worst} transport map, as shown by the following example:
	\begin{align*}
		X = Y = [-1, 1]^d,
		\quad 
		\mu = \nu = \tfrac{1}{2^d}\Lebesgue,
		\quad 
		V(x,y) = |x-y|^2,
		\quad 
		T(x) \assign -x,
		\quad 
		\nabla_X V(x, T(x)) = 4x.
	\end{align*}
	In this case the optimal transport map would correspond instead to $T^* : x \mapsto x$, and $T$ thus corresponds to the map \textit{maximizing} the total transport cost.
	However, since $\nabla_X V(x, T(x))$ is curl-free, the velocity field $v_t$ given by \eqref{eq:projection-AHT} is identically zero, and the trajectory $T_t$ with initial condition $T_0 = T$ remains stationary in an utterly suboptimal configuration. 

Except for special cases \cite{Brenier2009_AHT,SantambrogioOT}, convergence of the AHT scheme to the optimal transport map thus remains an open question.
	Up until now, the question of whether a suitable relaxation can make the AHT scheme unconditionally convergent has remained open as well. 
 We expect that regularization by diffusion for $\kappa>0$ helps the algorithm to escape from such configurations.

	\subsection{Outline}\hfill
	
	\textit{Existence of weak solutions.}
	We dedicate Section \ref{sec:preliminaries-ot} to recall the relevant optimal transport background. 
	Under suitable regularity assumptions, in 
	Section \ref{sec:minimizing-movement} we build an approximating sequence for the gradient flow of the energy $\nrg$ by means of a minimizing movement scheme and discuss the relevant a-priori estimates. Section \ref{sec:existence} shows convergence of such approximating sequences to a continuous limit trajectory $\rho$ with values on $\prob(X \times Y)$ and a pressure $\Pi$ in $L^2_{loc}(\Rnn; H^1(X))$. $(\rho, \Pi)$ satisfy the system of partial differential equations: 	
	\begin{equation}
	\int_0^\infty \intXY
	\partial_t \psi \dd \rho
	=
	\int_0^\infty \intXY
	\big(
	\nabla_X \psi\cdot\nabla_X [V - \Pi] 
	-
	\kappa
	\Delta_X \psi 
	\big)\dd \rho
	-
	\int_{X\times Y}
	\psi(0) \dd \rho^0
	\end{equation}
	for all $\psi \in C_c^\infty(\Rnn\times X \times Y)$ with $\nabla_X \psi \cdot \mathbf{n} \equiv 0$ on $\partial X$, and
	\begin{equation}
	\int_X \nabla \xi \cdot \nabla \Pi(t)\, \mu \dd x
	=
	\int_X \nabla \xi\cdot  \left[
	\int_Y \nabla_X V(x,y) \rhod(t,x, y) \dd \nu(y)
	\right]
	\dd x
	-\kappa \int_X \mu \Delta \xi \dd x
	\end{equation}
	for all $ \xi\in C^\infty(X)$ with $\nabla \xi \cdot \mathbf{n} = 0$ on $\partial X$ and a.e. $t\in\Rnn$, where $\rhod$ is the density of $\rho$ w.r.t $\Lebesgue \otimes \nu$.
	
	\textit{Stability of solutions.} 
	Section \ref{sec:stability} shows stability of weak solutions under convergence of the marginals $\mu$ and $\nu$ and of the initial condition $\rho^0$. 
	This relates the continuous $\nu$ limit to previous multiphasic flow studies such as \cite{OE97, CMN19, CM23}, identifying our dynamics as the limit of the multiphasic case when the mass of each phase $\nu$ is not a discrete measure supported on a finite set but rather a general probability measure. 
	
	\textit{Asymptotic convergence to the minimizer of $\nrg$.} 
	In Section \ref{sec:asymptotic} we show that weak solutions converge to the unique minimizer of the energy $\nrg$ as $t$ goes to infinity. The relevance of this result stems from the non-convexity of $\nrg$ in the fibered topology of the minimizing movement scheme. We also discuss implications to computational optimal transport. 
	
	\textit{The regime $\kappa = 0$.} 
	In Section \ref{sec:unregularized} we discuss the $\kappa = 0$ regime, and whether the results of Section \ref{sec:asymptotic} may apply to this limit. As explained in Section \ref{sec:motivation}, existence of solutions remains an open problem even in for the two-phase case, since the only source of regularity (the entropic term) in \eqref{eq:PDE1}--\eqref{eq:PDE2} is lost. 
	However, from the computational optimal transport perspective this regime is extremely interesting, as it would correspond to a relaxation of the AHT scheme to the space of transport \textit{plans}, while staying in the unregularized transport regime. 
	
	We show that, even though we cannot provide existence of solutions in the general case, the $\kappa= 0$ regime shares with the AHT scheme the existence of suboptimal stationary configurations.
	For instance, for $V(x,y) = |x-y|^2$, all initial conditions of the form $\rho^0 = (\id, \nabla\phi)_\sharp \mu$ with $\phi\in C^2(X)$ turn out to be stationary. 
	The intuition is that, for $\tau$ sufficiently small, $\rho^0$ becomes a stationary point of the minimizing movement scheme, and therefore also a stationary constant weak solution. 
	We conclude that a relaxation to the space of transport plans does not improve the convergence properties of the AHT scheme significantly, and that some kind of regularization is a necessary component for an unconditionally convergent variant of the AHT algorithm.
	
	\textit{Numerical simulations.} Finally, Section \ref{sec:numerics} exemplifies the stability and convergence results with numerical simulations of the minimizing movement scheme. 
	Experimentally we observe 
	linear convergence to the energy $\nrg$, which may open the door to novel numerical methods for the solution of entropic optimal transport in geometric domains.
	\subsection{General hypotheses}
	\label{sec:hypotheses}
	\begin{itemize}
		\item We assume $X \subset \R^d$ and $Y \subset \R^{d'}$ to be compact sets. $X$ is further assumed to have connected interior and a smooth boundary.
		\item We will employ the symbol $\mu$ for both a probability measure on $X$ and its density with respect to $\dd x$; the context should resolve any ambiguity. We will also assume that $\mu$ is bounded away from zero and from above, and that it has \textit{finite Fisher information}: 
		\begin{equation}
		\label{eq:fisher-info-mu}
		\frac14 \|\nabla \log \mu\|_{L^2(X,\mu)}^2
		= \int_X |\nabla \sqrt{\mu}|^2 \dd x < \infty.
		\end{equation} 
		Note that the boundedness away from zero and from above of $\mu$ makes the $L^2(X)$ and $L^2(X, \mu)$ norms (resp. $H^1(X)$ and $H^1(X, \mu)$ norms) equivalent, so we will often use them interchangeably. Finiteness of these norms for the marginal $\mu$ is needed to prove ``horizontal'' regularity of the solution $r$ in $H^1$. 
		\item We denote by $\nu$ a probability measure on $Y$.
		\item We use the symbol $\rho$ to denote a coupling between $\mu$ and $\nu$. Its Radon-Nykodym derivative with respect to $\Lebesgue \otimes \nu$, which (when it exists) is an element of $L^1(X\times Y, \Lebesgue\otimes \nu)$, is denoted by $r$. 
		We will use Lebesgue as the $X$-reference measure instead of $\mu$ because it simplifies the change of variables formula.
		Integration of a measurable function $f$ on $X$ with respect to the Lebesgue measure is simply denoted by $\int_X f \dd x$.
		\item $V$ will be a positive $C^2(X\times Y)$ function, and $\kappa$ will represent a positive parameter.
		\item Finally, we assume that the initial condition $\rho^0$ has finite energy $\nrg(\rho^0)$.
	\end{itemize}

	\section{Preliminaries}
	\label{sec:preliminaries-ot} 
	
	In the following chapters we will leverage the theory of optimal transportation to build an approximation sequence for the solutions of \eqref{eq:PDE1}--\eqref{eq:PDE2}. 
	For a thorough exposition of optimal transport we refer to \cite{SantambrogioOT,Villani-TOT2003}.
	
	For $Z\subset \R^d$ a compact set, denote by $\meas(Z)$ the set of Radon measures on $Z$, by $\measp(Z)$ the set of positive Radon measures on $Z$ and by $\prob(Z)$ the set of probability measures on $Z$. For $Z_1 \subset \R^d, Z_2\subset \R^{d'}$ compact sets and $\rho_1\in \prob(Z_1),\rho_2\in \prob(Z_2)$ two probability measures,  $\Gamma(\rho_1, \rho_2)$ denotes the set of \textit{couplings} or \textit{transport plans} between $\rho_1$ and $\rho_2$:
	\begin{equation}
	\label{eq:coupling}
	\Gamma(\rho_1, \rho_2) 
	\assign 
	\{
	\gamma\in\prob(Z_1\times Z_2)
	\ 
	\mid
	\ 
	\proj_1 \gamma = \rho_1,
	\quad
	\proj_2 \gamma = \rho_2   
	\}
	\end{equation}
	where the maps $\proj_1$ and $\proj_2$ denote the projections of measures on $Z_1\times Z_2$ to its marginals, i.e.
	\begin{align*}
	(\proj_1 \gamma)(S_1) \assign \gamma(S_1 \times Z_2) \qquad \tn{and} \qquad (\proj_2 \gamma)(S_2) \assign \gamma(Z_1 \times S_2)
	\end{align*}
	for $\gamma \in \prob(Z_1\times Z_2)$, $S_1 \subset Z_1$, $S_2 \subset Z_2$ measurable.
	
	The optimal transport problem between $\rho_1$ and $\rho_2$ with a continuous cost function $c: Z_1\times Z_2 \to \Rnn$ is then given by
	\begin{equation}
	\label{eq:ot}
	\min_{\gamma\in\Gamma(\rho_1, \rho_2)}\int_{Z_1\times Z_2} c(z_1,z_2) \dd \gamma(z_1, z_2).
	\end{equation}
	For compactly supported marginals and a lower-semicontinuous cost function $c$, Problem \eqref{eq:ot} admits a minimizer.
	
	When $d = d'$ and $c(z_1,z_2) = |z_1-z_2|^p$ with $p\in [1, \infty)$, the minimal score of \eqref{eq:ot} is (the $p$-th power of) the $L^p$ \textit{Wasserstein distance} (denoted by $\wass_p$) between $\rho_1$ and $\rho_2$. With no subscript, $\wass$ will be understood to denote the $\wass_2$ distance. 
	When $z_2 \mapsto c(z_1,z_2)$ is strictly convex for all $z_1\in Z_1$ (as is the case for the $L^p$ distance with $p>1$) and $\rho_1$ is Lebesgue-absolutely continuous, the unique minimizer of \eqref{eq:ot} is of the form $\gamma = (\id, T)_\sharp \rho_1$ \cite{MonotoneRerrangement-91,McCannGangboOTGeometry1996}, and $T$ solves the so-called \textit{Monge problem}:
	\begin{equation}
	\label{eq:monge}
	\min_{T_\sharp \rho_1 = \rho_2}\int_{Z1} c(z,T(z)) \dd \rho_1(z) .
	\end{equation}
	
	The optimal transport problem \eqref{eq:ot} admits the dual formulation: 
	
	\begin{equation}
	\label{eq:dual-ot}
	\sup_{\Phi(z_1) +  \Psi(z_2) \le c(z_1, z_2)}
	\int_{Z_1} \Phi(z_1)\dd \rho_1(z_1)  + \int_{Z_2} \Psi(z_2)\dd \rho_2(z_2) 
	\end{equation}
	where the supremum runs over continuous functions from $Z_1$ (resp. $Z_2$) to $\R$. 
	If $c$ is uniformly continuous maximizers for \eqref{eq:dual-ot} exist.
	
	When $Z_1 = Z_2 =: Z$, the $L^p$-Wasserstein distance for $p=1$ takes a particularly convenient form:
	\begin{equation}
	\label{eq:dual-wass1}
	\wass_1(\rho_1, \rho_2) 
	=
	\sup_{\text{Lip}(\Phi) \le 1}
	\int_Z \Phi(z)  \dd (\rho_1 - \rho_2)(z),
	\end{equation}
	which is known as the \textit{Kantorovich-Rubinstein formula}. On the other hand, for $p = 2$ it is customary to rewrite the dual relation \eqref{eq:dual-ot} as: 
	\begin{equation}
	\frac{1}{2} \wass(\rho_1, \rho_2)^2=\sup_{\Phi(z_1)+\Psi(z_2) \leq \frac{1}{2}|z_1-z_2|^2}
	\left( 
	\int_{Z_1} \Phi(z_1)\dd \rho_1(z_1)  + \int_{Z_2} \Psi(z_2)\dd \rho_2(z_2)
	\right).
	\label{eq:dual-wass2}
	\end{equation}
	This re-scaling allows to identify the optimal transport map $T$ as $T = \id - \nabla \Phi$ ($\rho_1$-almost everywhere). Besides, $\wass_1$ and $\wass_2$ satisfy an order relation: for $\gamma_1$ and $\gamma_2$ optimal transport plans between $\rho_1$ and $\rho_2$ for $\wass_1$ and $\wass_2$ respectively, it holds: 
	
	\begin{equation}
	\begin{aligned}
	\wass_1(\rho_0, \rho_1) 
	&= 
	\int_{Z^2} |z-z'| \dd \gamma_1(z,z')
	\le 
	\int_{Z^2}|z-z'| \dd \gamma_2(z,z')
	\\
	&\le
	\left(\int_{Z^2} |z-z'|^2 \dd \gamma_2(z,z')\right)^{1/2}
	=
	\wass_2(\rho_0, \rho_1),
	\label{eq:order-wass}
	\end{aligned}
	\end{equation}
	where the first inequality comes from the optimality of $\gamma_1$ and in the second we use Jensen's inequality. 
	
	A notion of `fiber-wise' Wasserstein distance will be central to our study. 
	Let $Z = X \times Y$ with $X \subset \R^d$ and $Y\subset \R^{d'}$ compacts and let $\nu \in \prob(Y)$. 
	For $\rho \in \prob(X\times Y)$ with $\proj_2 \rho = \nu$, define the \textit{disintegration of $\rho$ with respect to its second marginal} as the measurable selection of probability measures $(\rho_{y})_{y\in Y}$ satisfying:
	\begin{equation}
	\intXY \psi(x,y) \dd \rho(x,y)
	=
	\int_Y \left[
	\int_X \psi(x,y)
	\dd \rho_{y}(x)
	\right]
	\dd \nu(y)
	\end{equation}
	for all $\psi\in C(X\times Y)$.
	$\rho_y$ is uniquely defined $\nu$-almost everywhere.
	Then for $\rho_1, \rho_2\in \prob(X\times Y)$ with common second marginal, i.e.~$\proj_2 \rho_1 = \proj_2 \rho_2 = \nu$, define the \textit{fiber-wise Wasserstein distance} $\wassF$ by
	\begin{align}
	\wassF (\rho_1, \rho_2)^2 
	&:=
	\int_Y
	\wass(\rho_{1,y}, \rho_{2,y})^2 \dd \nu(y)
	\label{eq:wassF}
	\end{align}
	
	$\wassF$ is a metric by \cite{wass-fiber}, and for our case of study admits the following reformulations:
	\begin{lemma}
		\label{lemma:equivalence-formulation-wassF}
		If $\rho_1$ or $\rho_2$ is $\Lebesgue\otimes \nu$-absolutely continuous, the fiber-wise Wasserstein distance $\wassF$ admits the following equivalent formulations: 
		\begin{align}
			\wassF (\rho_1, \rho_2)^2 
			&=
			\inf_{
				\scriptsize
				\begin{array}{c}
					\gamma\in\measp(X\times X \times Y)
					\\
					\proj_{13} \gamma = \rho_1, \proj_{23} \gamma = \rho_2
				\end{array}
			}
			\int_{X\times X \times Y}
			|x-x'|^2 \dd \gamma(x,x', y)
			\label{eq:wassF1}
			\\
			&=
			\inf_{
				\scriptsize
				\begin{array}{c}
					\tilde{\gamma}\in\Gamma(\rho_1, \rho_2)
					\\
					y = y'\ \tilde{\gamma}\text{-a.e.}
				\end{array}
			}
			\int_{X\times Y\times X \times Y}
			|(x,y) - (x', y')|^2 \dd \tilde{\gamma}(x,y,x',y'),
			\label{eq:wassF2}
		\end{align}
	\end{lemma}

	\begin{remark}
		\label{remark:order-variables}
		Even though the different order in the variables of $\gamma$ and $\tilde{\gamma}$ may appear confusing at first, it answers to the different interpretation we attach to each formulation.
		Equation \eqref{eq:wassF1} should be interpreted in the sense that the $\wass$-optimal transport plan in each fiber (that we may denote as $\gamma_y$) can be wrapped into a joint measure $\gamma$.
		On the other hand, \eqref{eq:wassF2} exposes the relation between $\wassF$ and the Wasserstein distance on $\prob(X\times Y)$, being the former a more constrained version of the latter (and therefore yielding a higher score). 
	\end{remark}

	\begin{remark}
		\label{remark:selection}
		We make the density assumption in Lemma \ref{lemma:equivalence-formulation-wassF} for simplicity to avoid issues with measurability. Alternatively, one could employ a measurable selection theorem.
	\end{remark}

	\begin{proof}[Proof of Lemma \ref{lemma:equivalence-formulation-wassF}]
		For any feasible plan $\gamma$ in \eqref{eq:wassF1} one can define a feasible $\tilde{\gamma}$ in \eqref{eq:wassF2} with the same score as the pushforward of $\gamma$ by the map:
		\begin{equation}
			(x,x',y) \mapsto (x,y,x',y).
		\end{equation}
		Viceversa, for any feasible $\tilde{\gamma}$ in  \eqref{eq:wassF2}, its pushforward by  
		\begin{equation}
		(x,y,x',y') \mapsto (x,x',y).
		\end{equation}
		yields a feasible plan for \eqref{eq:wassF1}.
		For each $y\in Y$ let $\rho_{1,y}$ and $\rho_{2,y}$ denote respectively the $y$-disintegration of $\rho_1$ and $\rho_2$ w.r.t $\nu$. 
		Since either $\rho_{1}$ or $\rho_{2}$ have a density, the $L^2$-optimal transport plan between $\rho_{1,y}$ and $\rho_{2,y}$ is unique and a weak* continuous function of the marginals, and hence measurable. Let us denote such optimal plan by $\gamma_y^*$.
		Then one can define a candidate $\gamma^*$ for \eqref{eq:wassF1} by the measure acting as: 
		\begin{equation}
		\int_{X \times X \times Y} \phi(x, x', y)
		\dd \gamma^*(x, x', y)
		\assign 
		\int_Y 
		\left[
		\int_{X \times X}
		\phi(x, x', y)
		\dd \gamma_y^*(x, x')
		\right]
		\dd \nu(y).
		\end{equation}
		One can easily check that with this choice the right hand side of \eqref{eq:wassF1} equals the left hand side; therefore: 
		\begin{align}
			\wassF (\rho_1, \rho_2)^2 
			& \ge 
			\inf_{
				\scriptsize
				\begin{array}{c}
					\gamma\in\measp(X\times X \times Y)
					\\
					\proj_{13} \gamma = \rho_1, \proj_{23} \gamma = \rho_2
				\end{array}
			}
			\int
			|x-x'|^2 \dd \gamma(x,x', y).
		\end{align}
	
		For the reversed inequality, note that for any $\gamma$ feasible for the right hand side of \eqref{eq:wassF1}, it holds $\proj_3 \gamma = \nu$, and the disintegration of $\gamma$ with respect to $\nu$ (denoted again by  $\gamma_y$) provides a feasible transport plan between $\rho_{1,y}$ and $\rho_{2,y}$. 
		This means that, for $\nu$-a.e. $y\in Y$,
		\begin{equation}
			\wass(\rho_{1,y}, \rho_{2,y})^2
			\le 
			\int_{X\times X}
			|x-x'|^2 \dd \gamma_y(x,x'),
		\end{equation}
		so that integrating in $Y$ and taking the infimum yields the remaining inequality.
	\end{proof}

	 Unfortunately, the metric $\wassF$ does not turn the space of probability measures on $X\times Y$ into a compact metric space. The reason is that when $\nu$ does not have finite support, it is possible to build a sequence of measures on $X\times Y$ that are increasingly oscillating in the horizontal direction, so that the resulting sequence has no Cauchy subsequence (cf. \cite[Example 3.15]{wass-fiber}). As a consequence, the initial step of extracting a cluster point of the discrete trajectories of a gradient flow ---which for the $\wass$ metric is straighforward ---becomes an obstacle in the case of the $\wassF$ metric. 
	
	Nevertheless, $\wassF$ bounds the $\wass$ metric on the product space $X\times Y$ (see Remark \ref{remark:order-variables} and \cite[Proposition 3.10]{wass-fiber}, and therefore we will still be able to extract cluster points with respect to the standard Wasserstein distance on $X\times Y$. Adding this remark to \eqref{eq:order-wass} yields: 
	\begin{lemma}
		\label{lemma:order-wass}
		For $\rho_1, \rho_2$ probability measures on $X\times Y$ with identical $Y$-marginal it holds:
		\begin{equation}
		\wass_1(\rho_1, \rho_2) 
		\le
		\wass_2(\rho_1, \rho_2)
		\le 
		\wassF(\rho_1, \rho_2).
		\label{eq:order-wassF}
		\end{equation}
	\end{lemma}
	We will write $\rho \ll \Lebesgue \otimes \nu$ for a coupling $\rho\in \Gamma(\mu, \nu)$ that is absolutely continuous with respect to the measure $\Lebesgue \otimes \nu$. 
	This grants the existence of an $L^1$ density or \textit{Radon-Nykodim derivative}, that we will typically denote by $\rhod$. The following Lemma collects some useful calculus rules for absolutely continuous couplings: 
	
	\begin{lemma}
		\label{lemma:disintegrations-have-density}
		Let $\rho\in \Gamma(\mu, \nu)$ with $\mu \ll \dd x$ and $\rho \ll \Lebesgue \otimes \nu$. Denote by $r$ the density of $\rho$ with respect to $\Lebesgue \otimes \nu$. 
		Then: 
		\begin{enumerate}[label=\roman*)]
			\item For a.e.~$x\in X$ it holds: 
			\begin{equation}
				\int_Y \rhod(x,y) \dd \nu(y) = \mu(x),
			\end{equation}
			and the disintegration of $\rho$ w.r.t. its first marginal at $x$ is given by the measure $\tfrac{\rhod(x,\cdot)}{\mu(x)} \nu$.
			\item For $\nu$-a.e.~$y\in Y$ it holds: 
			\begin{equation}
			\int_X \rhod(x,y) \dd x = 1,
			\end{equation}
			and the disintegration of $\rho$ w.r.t. its second marginal at $y$ has a density given by $\rhod(\cdot, y)$.
		\end{enumerate}
	\end{lemma}
\begin{proof}
	i) For a continuous test function $\xi \in C(X)$,
	\begin{align*}
		\int_X \xi(x)
		\left[
		\int_Y \rhod(x,y) \dd \nu(y)
		\right]
		\dd x
		&=
		\intXY \xi(x) \rhod(x,y) \dd x \dd \nu(y)
		=
		\intXY \xi(x) \dd \rho(x,y)
		\\
		&=
		\intX \xi(x) \dd \mu(x)
		=
		\intX \xi(x) \mu(x) \dd x,
	\end{align*}
	and it is straightforward to verify that the disintegration of $\rho$ with respect to $\mu$ at $x\in X$ is given by $\tfrac{\rhod(x,\cdot)}{\mu(x)} \nu$; indeed, for any $\psi\in C(X\times Y)$,
	
	\begin{align*}
		\int_X 
		\left[
		\int_Y \psi(x,y)\frac{\rhod(x,y)}{\mu(x)} \dd \nu(y)
		\right] \dd \mu(x)
		&=
		\int_X 
		\left[
		\int_Y \psi(x,y)\frac{\rhod(x,y)}{\mu(x)} \dd \nu(y)
		\right]\mu(x) \dd x
		\\
		&=
		\intXY \psi(x,y) 
		\rhod(x,y) \dd x \dd \nu(y)
		=
		\intXY \psi(x,y) 
		 \dd \rho(x,y).
	\end{align*}
	An analogous computation yields ii).
\end{proof}

\begin{remark}
	In view of Lemma \ref{lemma:disintegrations-have-density}, if $\rho_1, \rho_2 \in \Gamma(\mu, \nu)$ are absolutely continuous w.r.t.~$\Lebesgue \otimes \nu$ with respective densities $\rhod_1$ and $\rhod_2$, their fibered Wasserstein distance may be denoted as: 
	
	\begin{equation}
		\wassF (\rho_1, \rho_2)^2 
		:=
		\int_Y
		\wass(\rhod_1(\cdot, y), \rhod_2(\cdot, y))^2 \dd \nu(y).
		\end{equation}
		\end{remark}
	
	The following lemma collects some useful facts about the structure of the minimizers of the energy $\nrg$; the proof can be found in \cite[Prop 2.5, (iii)]{DomDec} and references therein.
	The assumptions are not minimal, but adequate for our matter of study.

	\begin{lemma}[Minimizers of $\nrg$] Assume that $V\in C(X\times Y)$, and let $\mu \ll \Lebesgue$. Then:
		\label{lemma:entropic-minimizers}
		\begin{enumerate}
			\item $\nrg$ has a unique minimizer $\rho_* \in \adm$.
			\item There exist measurable $\Pi_* : X \to \R$ and $\Psi_* : Y \to \R$ such that
			\begin{equation}
			\rho_* = \exp\left(\frac{\Pi_* \oplus \Psi_* - V}{\kappa}\right)
			\Lebesgue \otimes \nu
			\end{equation} 
			\item Conversely, if $\Pi:X\to \R$ and $\Psi:Y\to\R$ are such that $\rho := \exp\left(\tfrac{\Pi \oplus \Psi - V}{\kappa}\right)\Lebesgue \otimes \nu$ belongs to $\adm$, then $\rho$ is the unique minimizer of $\nrg$.
		\end{enumerate}
	\end{lemma}

	During most of this article the marginals $\mu$ and $\nu$ will be fixed, and therefore the energy \eqref{eq:energy} and entropy \eqref{eq:entr} are defined without ambiguity. However, for some results (mainly Lemma \ref{lemma:stability-weak-solutions}) we will need to consider sequences $(\mu_n)_n$ and $(\nu_n)_n$ of marginals and respective couplings $\rho_n \in \Gamma(\mu_n, \nu_n)$. In that case it becomes practical to make the dependence of the energy in the marginals explicit, and hence define: 
	\begin{align}
		\label{eq:energy-mid}
		\nrg(\rho \mid \mu, \nu) 
		\assign  
		\int_{X\times Y} V(x,y) \dd \rho(x,y) + \kappa\entr(\rho\mid  \nu) +\indic_{\Gamma(\mu, \nu)}(\rho).
	\end{align}
	 The following Lemma will prove useful: 
	\begin{lemma}
		\label{lemma:energy-lsc}
		The marginal-dependent energy \eqref{eq:energy-mid} is lower-semicontinuous under joint weak* convergence of its arguments. 
	\end{lemma}
	\begin{proof}
		Let $(\mu_n)_n$ be a sequence in $\prob(X)$ converging weak* to $\mu$, and analogously $(\nu_n)_n$ be a sequence in $\prob(Y)$ converging weak* to $\nu$. Let $\rho_n\in \prob(X\times Y)$ for all $n$ converge to $\rho\in \prob(X\times Y)$.
		
		$\nrg(\rho \mid \mu, \nu)$ is a sum of three contributions. The first one is simply integration with respect to a fixed continuous function so it is continuous with respect to weak* convergence of $\rho_n$.
		The second entropic contribution is lower-semicontinuous with respect to joint weak* convergence of its first and second argument \cite{SavareGradientFlows}, as is the case here. 
		For the convex indicator function note the following: if $\rho_n \notin \Gamma(\mu_n, \nu_n)$ for all $n$, then $\nrg(\rho_n \mid \mu_n, \nu_n) = +\infty$ and there is nothing to prove. If however $\rho_n \in \Gamma(\mu_n, \nu_n)$ along some subsequence, then also $\rho \in \Gamma(\mu, \nu)$ by weak* continuity of the marginal projection operators. Therefore, the last contribution is also jointly weak* lower-semicontinuous with respect to its arguments.
	\end{proof}	
	
	\section{Discrete minimizing movements}
	\label{sec:minimizing-movement}
	We will construct an approximate solution via minimizing movements. To that end, fix a time step $\tau>0$ and define, for any given $\bar\rho\in\adm$,
	\begin{align}
	\label{eq:JKO-scheme}
	\nrg_\tau\big(\rho\mid \bar\rho\big)
	&:= \frac1{2\tau}\wassF(\rho, \bar{\rho})^2 + \nrg(\rho).
	\end{align}
	Starting from $\rho_\tau^0:=\rho^0$, define $\rho_\tau^n$ for $n=1,2,\ldots$ inductively as minimizer of the respective
	\mbox{$\nrg_\tau(\cdot\mid\rho_\tau^{n-1})$}.
	We will denote by $\rhod_\tau^n$ the density of $\rho^n_\tau$ with respect to $\Lebesgue \otimes \nu$.
	\begin{lemma}
		\label{lem:steppositive}
		The minimizing movement scheme \eqref{eq:JKO-scheme} is well-defined and yields a unique minimizer. 
		Moreover, $r_\tau^n>0$ for $\mu\otimes\nu$-almost every $(x,y)\in X\times Y$ and each $n\in \mathbb{N}$.
	\end{lemma}
	\begin{proof}
		We will show existence of minimizers of \eqref{eq:JKO-scheme} by constructing an auxiliary convex problem. 
		First let us show that the minimizing movement step \eqref{eq:JKO-scheme} can be equivalently written as:
		\begin{equation}
		\label{eq:primal}
		\begin{aligned}
		&\rho^n = \proj_{13}\gamma^n, 
		\quad
		\text{with}
		\\
		&\gamma^n
		=
		\argmin_{
			\scriptsize
			\begin{array}{c}
			\gamma\in\measp(X\times X \times Y)
			\\
			\proj_{1} \gamma = \mu, 
			\proj_{23} \gamma = \rho^{n-1}
			\end{array}
		}
		\int_{X\times X \times Y}
		\left(
		\frac{|x-x'|^2}{2\tau} + V(x,y) 
		\right)
		\dd \gamma(x,x',y)      
		+ \kappa
		\entr(\proj_{13} \gamma).
		\end{aligned}
		\end{equation}
		Intuitively, $\gamma$ represents a collection of transport plans on $(\gamma_y)_y \subset \meas_+(X\times X)$ indexed by the variable $Y$, the constraint $\proj_{23} \gamma = \rho^{n-1}$ grants that the second marginal of $\gamma_y$ is $\rho^{n-1}_y$, and the constraint $\proj_1 \gamma$ ensures that the new iterate $\rho^n$ has the right first marginal prescribed by $\mu$; the second marginal is automatically satisfied since no transport happens along the $Y$-axis and it is therefore inherited from $\rho^{n-1}$.
		Rewriting \eqref{eq:JKO-scheme} in this form will allow us to use the direct method of calculus of variations on the space $\measp(X\times X \times Y)$, simplifying the analysis.
		
		The proof of equivalence is very similar to that in Lemma \ref{lemma:equivalence-formulation-wassF}.
		First note that any feasible candidate $\rho$ in \eqref{eq:JKO-scheme} must be absolutely continuous with respect to $\Lebesgue \otimes \nu$, otherwise the entropy term would yield an infinite value. 
		Likewise, any feasible plan $\gamma$ in \eqref{eq:primal} must have a $\Lebesgue\otimes \nu$-absolutely continuous projection $\proj_{13}\gamma$.
		Now, for $\rho$ a feasible candidate for \eqref{eq:JKO-scheme}, and $\nu$-a.e.~$y\in Y$, define $\gamma_y$ as the $\wass$-optimal transport plan between the $y$-disintegration of $\rho$ and $\rho^{n-1}$ w.r.t $\nu$ (denoted respectively by $\rho_y$ and $\rho^{n-1}_y$). 
		By absolute continuity of $\rho$, Brenier's theorem grants uniqueness of such optimal transport plan.
		Then one can define a candidate $\gamma$ for \eqref{eq:primal} by the measure acting as: 
		\begin{equation}
			\int_{X \times X \times Y} \phi(x, x', y)
			\dd \gamma(x, x', y)
			\assign 
			\int_Y 
			\left[
			\int_{X \times X}
			\phi(x, x', y)
			\dd \gamma_y(x, x')
			\right]
			\dd \nu(y).
		\end{equation}
		
		One can easily check that the value $\nrg_\tau(\rho \mid \rho^{n-1})$ is equal to that of the primal objective \eqref{eq:primal} applied to $\gamma$. 
		Conversely, if $\gamma$ is a feasible candidate for the primal problem \eqref{eq:primal}, its score is necessary larger or equal than $\nrg_\tau(\proj_{13}\gamma \mid \rho^{n-1})$, since the $V$ and $\entr$ terms yield the same value, and the transport part is optimal in \eqref{eq:JKO-scheme}.
		We conclude that the minimization of $\nrg_\tau(\cdot \mid \rho^{n-1})$ is equivalent to the convex problem \eqref{eq:primal} (i.e.~they share the same optimal values, and if there are minimizers for one of them we showed how to construct a minimizer for the other).
		
		Next, all the terms constituting the functional in \eqref{eq:primal} are lower-semicontinuous in $\gamma$. By the marginal constraints any admissible $\gamma$ lies in $\prob(X\times X \times Y)$, which is pre-compact under the weak* topology (by compactness of $X$ and $Y$). Therefore the direct method in the calculus of variations yields existence of a minimizer $\gamma^n$.
		Besides, the strict convexity of the entropy $\entr$ guarantees that $\rho^n  \assign \proj_{13} \gamma^n$ is unique, which yields a unique minimizer for $\nrg_\tau(\cdot \mid \rho^{n-1})$.
		
		To show positivity of the newly obtained $\rho^n$, let $\veps \in (0,1)$ and consider the feasible coupling $\rho_\veps \assign (1-\veps)\rho^n + \veps \mu\otimes \nu$. We will show that if $\rho^n$ were to vanish in a set $A\subset X\times Y$ with $(\mu\otimes \nu)(A)>0$ then there would exist a value of $\veps>0$ such that $\nrg_\tau(\rho_\veps\mid \rho^{n-1}) < \nrg_\tau(\rho^n\mid \rho^{n-1})$, which would contradict the optimality of $\rho^n$. 
		To that end let us bound the upper derivative of $\nrg_\tau(\rho_\veps\mid \rho^{n-1})$ with respect to $\veps$ (denoted by $\DU \nrg_\tau(\rho_\veps\mid \rho^{n-1})$) at $\veps = 0$. First, denote by $\gamma$  (resp.~$\gamma_\otimes$) 
        a feasible transport plan between $\rho^n$ (resp.~$\mu\otimes \nu$) and $\rho^{n-1}$ for problem \eqref{eq:wassF1}.
        Then one can build a $\wassF$-feasible transport plan between $\rho_\veps$ and $\rho^{n-1}$ as $\gamma_\veps \assign (1-\veps)\gamma + \veps \gamma_\otimes$. Consequently: 
		\begin{equation}
		\wassF(\rho_\veps, \rho^{n-1})^2
		\le 
		(1-\veps) \wassF(\rho^n, \rho^{n-1})^2 + \veps(\diam X)^2
		\le 
		\wassF(\rho^n, \rho^{n-1})^2 + \veps(\diam X)^2,
		\end{equation}
		so $\DU\wassF(\rho_\veps, \rho^{n-1})^2\big|_{\veps = 0} \le (\diam X)^2$. For the potential term, by linearity: 
		\begin{equation}
		\int_{X \times Y} V \dd \rho_\veps 
		=
		(1-\veps)\int_{X \times Y} V \dd \rho^n 
		+
		\veps\int_{X \times Y} V \dd (\mu\otimes \nu) 
		\le 
		\int_{X \times Y} V \dd \rho^n 
		+
		\veps\|V\|_\infty,		
		\end{equation}
		and thus $\DU \la V , \rho_\veps\ra \big|_{\veps = 0} \le \|V\|_\infty$.
		Finally, for the entropic term, let 
		\mbox{$\entr(\rho) = \intXY f(r) \dd x \dd \nu(y)$}, with $f(s) \assign s \log s - s + 1$.
		$f$ is a convex function with derivative $f'(s) = \log s$.
		Since $f$ is convex, the finite difference $(f(s + h) - f(s)) / h$ converges monotonously to $f'(s)$ as $h\rightarrow 0$. 
		This allows us to use the monotone convergence theorem in the following calculation
		\begin{align*}
		\frac{\dd}{\dd \veps}
		\entr(\rho_\veps) \big|_{\veps = 0}
		&=
		\lim_{\veps \rightarrow 0}
		\frac1\veps (\entr(\rho_\veps)
		- \entr(\rho))
		=
		\lim_{\veps \rightarrow 0}
		\intXY
		\frac1\veps 
		[f(\rhod^n + \veps (\mu - \rhod^n)) - f(\rhod^n)]
		\dd x \dd \nu(y)	
		\\
		&=
		\int_{X \times Y}
		\log(\rhod^n(x,y)) 
		[\mu(x) - \rhod^n(x,y)]
		\dd x \dd \nu(y)
		\\
		&=
		\int_{X \times Y}
		\log(\rhod^n(x,y)) \dd (\mu\otimes \nu)(x,y)
		- \entr(\rho^n) + \intXY (1-\rhod^n(x,y)) \dd x \dd \nu(y)
		\\
		&\le
		\int_{X \times Y}
		\log(\rhod^n(x,y)) \dd (\mu\otimes \nu)(x,y)
		+|X| - 1.
		\end{align*}
		where the last step holds by the positivity of the entropy.
		Thus, defing $C = (\diam X)^2 /(2\tau)+ \|V\|_\infty + |X| - 1 <\infty$ we obtain
		
		\begin{equation}
		\overline{D}_\veps
		\nrg_\tau(\rho_\veps\mid \rho^{n-1}) \big|_{\veps=0}
		\le 
		C 
		+
		\int_{X \times Y}
		\log(\rhod^n(x,y)) \dd (\mu\otimes \nu)(x,y),
		\end{equation}
		which by optimality of $\rho^n$ implies that $\rhod^n$ must be positive ($\mu\otimes\nu)$-almost everywhere.
	\end{proof}
	
	\begin{lemma}
		\label{lem:equicontinuity-rho}
		The iterates satisfy
		\begin{equation}
		\wass(\rho_\tau^m, \rho_\tau^n) 
		\le
		\wassF(\rho_\tau^m, \rho_\tau^n) 
		\le 
		C\sqrt{\tau\,|m-n|}
		\label{eq:equicontinuity-rho}
		\end{equation}
		with $C =\sqrt{2\nrg(\rho^{0})}$.
	\end{lemma}
	\begin{proof}
		\eqref{eq:equicontinuity-rho} follows from the classical JKO energy estimates applied to \eqref{eq:JKO-scheme}: choosing $\rho^{n-1}_\tau$ as a competitor in \eqref{eq:JKO-scheme}, one gets:
		
		\begin{equation}
		\label{eq:energy-estimate}
		\wassF(\rho_\tau^n, \rho_\tau^{n-1})^2
		\le 
		2\tau (\nrg(\rho_\tau^{n-1}) - \nrg(\rho_\tau^n)).
		\end{equation}
        W.l.o.g.~let $n < m$. Then the triangle inequality for the fiber metric $\wassF$ followed by Jensen's inequality yields:
		\begin{align*}
		\wassF(\rho_\tau^m, \rho_\tau^{n})^2
		&\le 
		\left(
		\sum_{k = n+1}^m
		\wassF(\rho_\tau^k, \rho_\tau^{k-1})
		\right)^{2}
		\le 
		(m-n)
		\sum_{k = n+1}^m
		\wassF(\rho_\tau^k, \rho_\tau^{k-1})^2.
		\\
		\intertext{Finally, using the energy estimate \eqref{eq:energy-estimate}:}
		&
		\le 
		2(m-n) \tau (\nrg(\rho_\tau^{n}) - \nrg(\rho_\tau^m))
		\le 
		2\nrg(\rho^{0})
		\tau (m-n),
		\end{align*}
		and the bound in $\wass$ follows from Lemma \ref{lemma:order-wass}, non-negativity of $\nrg$, and the finiteness of the initial energy.
	\end{proof}

	\begin{lemma}
		\label{lemma:bound-X-derivative-rho}
		For each $n=1,2,\ldots$, there is a map $\Pi_\tau^n \in H^1(X)$ such that, for $\nu$-a.e. $y\in Y$,
		\begin{equation}
		\label{eq:push-forward-JKO}
		\rho^{n-1}_{\tau,y}
		 = \big(\id_X + \tau\nabla_X\big[V(\cdot, y)
		-
		\Pi^n_\tau
		+ \kappa\log\rhod_\tau^n(\cdot, y)
		\big]\big)_\sharp\rho^n_{\tau, y}.
		\end{equation}
		$\Pi^n_\tau$ can be chosen to have zero mean with respect to $\mu$ for all $n$ and $\tau$. Besides, it holds
		\begin{equation}
		\label{eq:velocity}
		\int_{X\times Y}
		|\nabla_X(V - \Pi_\tau^n + \kappa\log \rhod_\tau^n)|^2
		\dd \rho^n_\tau
		=
		\frac{\wassF(\rho^n,\rho^{n-1}\big)^2 }{\tau^2}.
		\end{equation}
		In particular, for any $0 \le k_1	< k_2$,  \begin{equation}
		\label{eq:dissipation}
		\nrg(\rho_\tau^{k_1})
		+
		\sum_{n = k_1+1}^{k_2}
		\tau
		\int_{X\times Y}
		|\nabla_X(V - \Pi_\tau^n + \kappa\log \rhod_\tau^n)|^2
		\dd \rho^n_\tau
		\le 
		\nrg(\rho^0).
		\end{equation}
		
		Moreover, $\rho_\tau^n$ is bounded away from zero and from above with a uniform bound depending on $\kappa$ and $\tau$ but not on $n$. 
		Furthermore, for each $N\in \mathbb{N}$:
		\begin{align}
		\label{eq:L2-bound-phi}
		\tau \sum_{n = 1}^{N} \left( 
            \kappa^2\int_{X\times Y}|\nabla_X\sqrt{r_\tau^n}|^2\dd x\dd\nu(y)
		      +\int_X \big|\nabla\Pi^n_\tau\big|^2 \dd \mu
        \right)
		\le
		4\nrg(\rho^0) + C \tau N,
		\end{align}
		with $C$ a constant depending only on $\|\nabla_X V\|_\infty$, $\kappa$, and the Fisher information of $\mu$.
	\end{lemma}
	\begin{proof}
		In this proof we will drop the subscript $\tau$ for better readability.
		
       We extend the method developed in \cite{SantambrogioOT}. Let $\Phi:X\times Y\to\R$ be a measurable function such that, for $\nu$-a.e.~$y\in Y$, the map $x\mapsto \Phi(x,y)$ is the optimal dual potential in \eqref{eq:dual-wass2} between $\rho^n_y$ and $\rho^{n-1}_y$, i.e.,
        \begin{align}
		      \label{eq:Phipush}
		      (\id_X - \nabla_X\Phi(\cdot,y))_\sharp\rho^n_y = \rho^{n-1}_y
		\end{align}
        for those $y$. For $\nu$-a.e.~$y\in Y$, the density $\rhod^n(\cdot,y)$ is positive a.e.~on $X$ by Lemma \ref{lem:steppositive}, hence $\Phi(\cdot,y)$ is Lipschitz continuous for those $y$, with a $y$-uniform Lipschitz constant given by $\diam(X)$.
        Further, 
        $\Phi(\cdot,y)$ is unique up to an additive constant, again thanks to positivity of $\rho^n$ and the connectedness of $X$; we assume that $\Phi(\cdot,y)$ is normalized to have zero average (that normalization will be changed at the end of the proof). Finally, as outlined in Section \ref{sec:preliminaries-ot}, optimality implies that
		\begin{align}
		\frac12\wass\big(\rhod^n(\cdot, y),\rhod^{n-1}(\cdot, y)\big)^2 
		&= 
		\frac12\int_X \big|\nabla_X\Phi(x,y)\big|^2 \rhod^n(x,y)\dd x, \\
		\label{eq:wass2}
		&= \int_X \big[\Phi(x,y)\rhod^n(x,y) + \Psi(x,y)\rhod^{n-1}(x,y)\big]\dd x,
		\end{align}
		for $\nu$-a.e.~$y\in Y$, where $\Psi(\cdot,y)$ is the corresponding Kantorovich potential (cf.~\eqref{eq:dual-wass2}).
		
        Now let $\theta\in L^\infty(X\times Y, \Lebesgue \otimes \nu)$ be a perturbation satisfying
		\begin{align}
		      \label{eq:admperturb}
		      \int_X\theta(x,y) \rhod^n(x,y)\dd x = 0 \quad \text{for $\nu$-a.e.~$y\in Y$},
		      \quad 
		      \int_Y\theta(x,y) \rhod^n(x,y)\dd\nu(y) = 0 \quad \text{for a.e.~$x\in X$}.
		\end{align}
		  Then  $\rho_\eps:=(1+\eps\theta)\rho^n\in\adm$, at least for sufficiently small $\eps\ge0$, i.e., $\eps<1/\|\theta\|_{L^\infty}$, since \eqref{eq:admperturb} guarantees the conservation of the marginals $\mu$ and $\nu$. 
		  Notice further that the corresponding density $\rhod_\eps(\cdot,y)$ is positive for $\nu$-a.e.~$y\in Y$.  Let $\Phi_\eps$ and $\Psi_\eps$ be such that $\Phi_\eps(\cdot,y)$ is the Kantorovich potential from $\rhod_\eps(\cdot,y)$ to $\rhod^{n-1}(\cdot,y)$ for $\nu$-a.e.~$y\in Y$, made unique by normalization to zero average, and $\Psi_\eps(\cdot,y)$ is the respective $c$-transform. Again by positivity of the densities, and by boundedness of $X$, we may conclude that $\Phi_\eps(\cdot, y)\to\Phi(\cdot, y)$ uniformly as $\eps\to0$ for $\nu$-a.e.~$y\in Y$
		  \cite[Theorem 1.52]{FS18}. 
		  By optimality of $\rho^n$ in $\nrg(\cdot\mid\rho^{n-1})$, and of $\Phi$ and $\Psi$ for the transport, we then have
		\begin{align*}
		&\frac1\tau \int_{X\times Y} \big[\Phi_\eps(x,y)\rhod_\eps(x,y) + \Psi_\eps(x,y)\rhod^{n-1}(x,y)\big]\dd x \dd\nu(y) + \nrg(\rho_\eps) \\ 
		&\ge \frac1\tau\int_{X\times Y} \big[\Phi(x,y)\rhod^n(x,y) + \Psi(x,y)\rhod^{n-1}(x,y)\big]\dd x\dd\nu(y)  + \nrg(\rho^n) \\
		&\ge\frac1\tau\int_{X\times Y} \big[\Phi_\eps(x,y)\rhod^n(x,y) + \Psi_\eps(x,y)\rhod^{n-1}(x,y)\big]\dd x\dd\nu(y)  + \nrg(\rho^n).
		\end{align*}
		After subtraction and division by $\eps>0$, it follows that
		\begin{align*}
		\frac1\tau \int_{X\times Y} \Phi_\eps(x,y)
		(\rhod^n\theta)(x,y)\dd x\dd\nu(y) 
		& + \int_{X\times Y} V(x,y)
		(\rhod^n\theta)(x,y)\dd x\dd\nu(y) \\ 
		& + \frac\kappa\eps\int_{X\times Y} \big(\rhod_\eps\log\rhod_\eps-\rhod_\eps - \rhod^n\log\rhod^n+\rhod^n\big)\dd x\dd\nu(y)  \ge 0.
		\end{align*}
		Passing to $\eps\to0$, recalling that $\Phi_\eps(\cdot, y)\to\Phi(\cdot, y)$ uniformly and $\Phi$ is uniformly bounded by boundedness of the domain, we can extract the limit of the first two terms by means of dominated convergence.
		For the third term, since $s \mapsto f(s) = s\log s - s + 1$ is a differentiable convex function with $f'(s)=\log s$, it holds pointwise that $f(\rhod_\veps) - f(\rhod^n) \ge \veps \theta \rhod^n \log \rhod^n$, which in the $\veps \to 0$ limit yields:
		\begin{align*}
		\int_{X\times Y} \left[\frac1\tau\Phi(x,y)+V(x,y)+\kappa\log\rhod^n(x,y)\right]
		\theta(x,y)\dd \rho^n(x,y)
		\ge 0.
		\end{align*}
		Given that $\theta\in L^\infty(X\times Y, \Lebesgue \otimes \nu)$ is arbitrary subject to \eqref{eq:admperturb}, we conclude that there are functions $P\in L^1(X, \mu)$ and $Q\in L^1(Y, \nu)$ such that
		\begin{equation}
		\label{eq:euler-lagrange-pi}
		\frac1\tau \Phi(x,y) + V(x,y) + \kappa \log\rhod^n(x,y) = P(x) + Q(y) ;
		\end{equation}
        see Lemma \ref{lem:HB} from the Appendix for details.       
		Since $Q$ only shifts the level of the potentials $\Phi(\cdot,y)$ ---which does not affect the $x$-gradient---, \eqref{eq:push-forward-JKO} and \eqref{eq:velocity} follow in view of \eqref{eq:Phipush} with $\Pi^n:=P$. 
		$P$ can be chosen to have zero mean with respect to $\mu$ by adding an additional global constant offset to $\Phi$. Moreover, summing \eqref{eq:velocity} in time and using \eqref{eq:energy-estimate} yields \eqref{eq:dissipation}.
		
		Let us now show that $\rhod^n$ is uniformly bounded from above and away from zero, with a bound that does not depend on the iteration number $n$. The function $\Phi/\tau + V$ is uniformly Lipschitz
		in $x$ by \eqref{eq:Phipush}, with Lipschitz constant $L \assign \diam(X) / \tau + \|\nabla V \|_\infty$, so \eqref{eq:euler-lagrange-pi} implies that $\kappa \log \rhod^n - \Pi^n$ must be uniformly $L$-Lipschitz in $x$ as well. To deduce the individual regularity of each $\rhod^n$ and $\Pi^n$, define first the auxiliary function $\sigma \assign \rhod^n e^{-\Pi^n/\kappa}$. Then $\kappa \log \sigma = \kappa \log \rhod^n - \Pi^n$, so $\kappa \log \sigma$ is $L$-Lipschitz in $x$. Now note that, since $|\nabla_X (\kappa \log \rhod^n - \Pi^n)|$ is uniformly bounded by $L$, after averaging in $Y$ the Lipschitz bound cannot be worse. 
		Therefore we find that $\mu$-almost everywhere in $X$:
		\begin{align*}
		L \mu
		&\ge
		\left|
		\int_Y \nabla_X (\kappa \log \rhod^n - \Pi^n) \rhod^n \dd \nu(y)
		\right|
		=
		\left|
		\kappa \int_Y (\nabla_X \log \sigma ) \sigma e^{\Pi^n/\kappa} \dd \nu(y)
		\right|
		=
		\left|
		\kappa \int_Y (\nabla_X \sigma) e^{\Pi^n/\kappa} \dd \nu(y)	 
		\right|		
		\\
		&
		=
		\left|
		\kappa \nabla_X \left( \int_Y\sigma \dd \nu(y)\right)
		e^{\Pi^n/\kappa}		 		
		\right|
		=
		\left|
		\kappa \nabla \left( \mu e^{-\Pi^n/\kappa}	\right)
		e^{\Pi^n/\kappa}		 		
		\right|
		=
		\mu 
		|\nabla (\kappa\log \mu - \Pi^n)|.
		\end{align*}
		This shows that the function $\Pi^n - \kappa \log \mu$ is also $L$-Lipschitz, and since by assumption $\nabla\log\mu$ is an $L^2$ function, $\nabla \Pi^n$ is in $L^2$ as well.
		
		We will now show the uniform lower bound of $\rhod^n$.
		Define $\Psi(x, y) = \kappa \log r^n(x,y) -  \kappa \log \mu(x)$, which by \eqref{eq:euler-lagrange-pi} can be rewritten as: 
		\begin{equation}
			\Psi(x,y)
			=
			\Pi^n(x)  - \kappa \log \mu - \frac{1}{\tau} \Phi(x,y) - V(x,y).
		\end{equation}  
		Therefore $\Psi$ has an $x$-Lipschitz constant of $2L$.
		On the other hand note that, by the properties of disintegrations (Lemma \ref{lemma:disintegrations-have-density}), for $\nu$-a.e.~$y\in Y$ it holds:
		\begin{equation}
		1
		=
		\int_X r(x,y) \dd x
		=
		\int_X  \frac{r(x,y)}{\mu(x)} \dd \mu(x )
		=
		\int_X 	e^{\Psi(x,y)/\kappa} \dd \mu(x),
		\end{equation}
		and consequently:
		\begin{equation}
		\label{eq:losumexp}
		0 = \kappa \log \int_X 	e^{\Psi(x,y)/\kappa} \dd \mu(x).
		\end{equation}
		Since the log-sum-exp behaves like a weighted average, \eqref{eq:losumexp} is between the maximum and minimum value of $\Psi$; and since $\Psi$ is continuous in $x$ and $X$ is connected, we conclude that for $\nu$-a.e.~$y\in Y$ there exists a point $\hat{x}\in X$ such that $\Psi(\hat{x}, y) = 0$. As a result, $|\Psi|$ is uniformly bounded by $2L\, \diam X $.
		We conclude that there exists a finite constant $C>0$ depending on $\diam(X)$, $\tau$, $\|\nabla V\|_{\infty}$, $\kappa$ and the lower and upper bounds of $\mu$ (but not on $n$) such that $\rhod^n = \mu \exp(\Psi/\kappa)$ is bounded as: 
		\begin{equation}
		e^{-C }
		\le 
		r^n(x,y)
		\le 
		e^{C }.
		\end{equation}
		
		For the estimate \eqref{eq:L2-bound-phi} let us further develop \eqref{eq:velocity}:
		\begin{align*}
		\frac{1}{\tau^2}&\wassF(\rho^n,\rho^{n-1}\big)^2 
		= 
		\int_{X\times Y} \big|\nabla_X(V - \Pi^n + \kappa\log \rhod^n)\big|^2 \dd \rho^n
		\\
		&=
		\int_{X} \big|\nabla \Pi^n \big|^2 \dd \mu
		+
		\int_{X\times Y} \big|\nabla_X(V  + \kappa\log \rhod^n)\big|^2 \dd\rho^n
		-
		2 \int_{X\times Y} \nabla \Pi^n \cdot \nabla_X(V  + \kappa\log \rhod^n) \dd\rho^n.
		\end{align*}
		Rearranging terms, and noting that the second contribution is always non-negative:
		\begin{equation}
		\int_{X\times Y} \big|\nabla \Pi^n \big|^2 \dd\mu
		\le 
		\frac{1}{\tau^2}\wassF(\rho^n,\rho^{n-1}\big)^2 
		+
		2 \int_{X\times Y} \nabla \Pi^n \cdot 
		\nabla_X\left( V + \kappa\log \rhod^n\right)\dd\rho^n.
		\label{eq:bound-pi-one-step}
		\end{equation}
		Note that for the entropic term we can write:
		\begin{align*}
		\int_{X\times Y} \nabla \Pi^n \cdot 
		\nabla_X \log \rhod^n\dd \rho^n
		&=
		\int_{X\times Y} \nabla \Pi^n \cdot 
		\frac{\nabla_X \rhod^n}{\rhod^n} \dd\rho^n
		=
		\int_{X\times Y} \nabla \Pi^n \cdot 
		\nabla_X \rhod^n \dd x \dd \nu(y)
		\\
		&=
		\int_{X} \nabla \Pi^n \cdot 
		\nabla \mu \dd x
		=
		\int_{X} \nabla \Pi^n \cdot 
		\nabla \log \mu \dd \mu.
		\end{align*}
		Thus, summing \eqref{eq:bound-pi-one-step} in time:
		
		\begin{align*}
		\sum_{n = 1}^{N}
		\tau
		&\int_{X\times Y} \big|\nabla \Pi^n \big|^2 \dd \mu(x)
		\le 
		\sum_{n = 1}^{N}
		\frac{1}{\tau}\wassF(\rho^n,\rho^{n-1}\big)^2
		+ 2\tau\int_{X\times Y} \nabla \Pi^n \cdot 
		\nabla_X\left( V + \kappa\log \rhod^n\right)\dd \rho^n ,
		\end{align*}
		and using \eqref{eq:energy-estimate} for the first term:
		\begin{align*}
		&\le
		2\nrg(\rho^0)
		+
		\sum_{n = 1}^{N}
		2\tau\int_{X\times Y} \nabla \Pi^n \cdot 
		\nabla_X V \dd \rho^n 
		+
		2\tau\kappa\int_{X} \nabla \Pi^n \cdot 
		\nabla \log \mu \dd \mu
		\\
		&\le
		2\nrg(\rho^0)
		+
		\sqrt{ N \tau}
		\big(2 \|\nabla_X V\|_\infty +
		2\kappa \|\nabla \log\mu\|_{L^2(X, \mu)}\big)
		\left(
		\sum_{n = 1}^{N}
		\tau
		\int_{X} \big|\nabla \Pi^n \big|^2 \dd \mu
		\right)^{1/2}
		\\
		&\le 
		2\nrg(\rho^0)
		+
		\frac12 N \tau
		\big(2 \|\nabla_X V\|_\infty +
		2\kappa \|\nabla \log\mu\|_{L^2(X, \mu)}\big)^2
		+
		\frac12
		\sum_{n = 1}^{N}
		\tau
		\int_{X} \big|\nabla \Pi^n \big|^2 \dd \mu.
		\end{align*}
		The bound on $\Pi^n$ in \eqref{eq:L2-bound-phi} now follows from cancelling half of the sum in time and accounting for \eqref{eq:fisher-info-mu}. To obtain the bound on $\sqrt{\rhod^n}$, we expand the derivative of $\Phi$ in a different way:
        \begin{align*}
            \frac{1}{\tau}\intXY |\nabla_X\Phi|^2 \dd \rho^n 
            &= 
            \tau\intXY|\nabla_X(\kappa\log\rhod^n + V - \Pi^n)|^2 \dd \rho^n
            \\
            &\ge 
            \tau
            \intXY
            \left[
            \frac{\kappa^2}2| \nabla_X\sqrt{\rhod^n}|^2 - 2|\nabla_X V|^2\rhod^n - 2|\nabla\Pi^n|^2\rhod^n
            \right]
            \dd x \dd \nu(y).
        \end{align*}
        Using the already established estimate on $\nabla\Pi^n$, the remaining part of \eqref{eq:L2-bound-phi} follows.
	\end{proof}
 
	\begin{lemma}
		For every $\psi\in C^\infty(\overline X\times Y)$ with $\nml\cdot\nabla_X\psi(x,y)=0$ for all $x\in\partial X$, $y\in Y$, 
		we have that
		\begin{align}
		\label{eq:weakxz}
		\int_{X\times Y} \psi \dd \left[
		\frac{\rho_\tau^n-\rho_\tau^{n-1}}\tau
		\right]
		= 
		-\int_{X\times Y} \nabla_X\psi\cdot\nabla_X\big[V-\Pi_\tau^n\big]\dd \rho_\tau^n 
		+ 
		\kappa\int_{X\times Y} \Delta_X\psi \dd \rho_\tau^n
		+ 
		\tau\,R_\tau^n(\psi),
		\end{align}
		where the remainder is bounded as
		\begin{align} 
		\label{eq:remainder}
		\big|R_\tau^n(\psi)\big|\le \frac12\|\psi\|_{C^2}S_\tau^n,
		\quad S_\tau^n:=\frac{\wassF(\rho_\tau^n, \rho_\tau^{n-1})^2}{\tau^2}.
		\end{align}
		In particular, for every $\xi\in C^\infty(\overline X)$ with $\nml\cdot\nabla_X\xi(x)\equiv0$ on $\partial X$,
		\begin{align}
		\label{eq:weakx}
		\int_X \nabla_X\xi\cdot\nabla_X\Pi_\tau^n\,\dd \mu - \int_{X\times Y} \nabla_X\xi\cdot\nabla_XV\,\dd \rho_\tau^n
		+ 
		\kappa\int_{X\times Y}
		\Delta_X \xi \dd \mu(x)
		= \tau\, R_\tau^n(\xi).
		\end{align}
	\end{lemma}
	\begin{proof}
		Integrating $\rho_\tau^{n-1}$ with respect to $\psi$ and using \eqref{eq:push-forward-JKO} to perform:
		\begin{align*}
		\int_{X\times Y}
		\psi(x,y)
		\dd\rho_\tau^{n-1}(x,y)
		&=
		\int_{X\times Y}
		\psi(x + \tau \nabla_X[V - \Pi_\tau^n + \kappa\log \rhod_\tau^n],y)
		\rhod^n_\tau(x,y) \dd x \dd \nu(y)
		\\
		&=
		\int_{X\times Y}
		[
		\psi(x,y)
		+
		\tau
		\nabla_X\psi\cdot \nabla_X(V - \Pi_\tau^n + \kappa\log \rhod_\tau^n)
		]
		\rhod_\tau^n   \dd x \dd \nu(y)  
		\\   
		&\ \ + 
		\frac{\tau^2}{2} R_\tau^n(\psi),
		\end{align*}
		where the remainder $R_\tau^n(\psi)$ is bounded by the second-order term: 
		\begin{align*}
			|R_\tau^n(\psi)| 
			&\le 
			\|\nabla_X^2\psi\|_\infty \int_{X\times Y}
			|\nabla_X(V - \Pi_\tau^n + \kappa\log \rhod_\tau^n)|^2
			 \dd \rho^n_\tau
			\le \|\nabla_X^2\psi\|_\infty \int_{X\times Y}
			\frac{\wassF(\rho_\tau^n, \rho_\tau^{n-1})^2}{\tau^2}.
		\end{align*}
	
		Then \eqref{eq:weakxz} follows by collecting the finite difference in $\rho^n_\tau$ and integrating the entropic term by parts, while \eqref{eq:weakx} results from choosing $\psi(x,y)=\xi(x)$.
	\end{proof}
	
	\section{Existence of weak solutions}
	\label{sec:existence}
	Throughout this section we denote by $\bar{\Pi}_\tau$ and $\bar{\rho}_\tau$ the piecewise-constant interpolations of the iterates $(\Pi^n_\tau)_n$ and $(\rho^n_\tau)_n$:
	\begin{align}
	\bar{\rho}_\tau(t) 
	&=
	\rho_\tau^{\lfloor t/\tau \rfloor}, 
	\qquad 
	\bar{\rhod}_\tau(t) 
	=
	\rhod_\tau^{\lfloor t/\tau \rfloor}, 
	\qquad 
	\bar{\Pi}_\tau(t) 
	=
	\Pi_\tau^{\lfloor t/\tau \rfloor}, 
	\qquad
	\text{for all $t \in [0, \infty)$}.
	\end{align}
	
	The main result is their convergence to weak solutions of the evolution equation \eqref{eq:PDE1}-\eqref{eq:PDE2}. 
	There is a technical aspect worth noting.  
	Although most of the terms in the time-integrated version of \eqref{eq:weakxz} can be shown to converge by sheer weak* compactness of $\bar{\rho}_\tau$, the pairing between $\bar{\Pi}_\tau$ and $\bar{\rho}_\tau$ poses a special difficulty, since both converge only weakly. The control on the horizontal derivatives of $\bar{\rho}_\tau$ provided by Lemma \ref{lemma:bound-X-derivative-rho} is unfortunately insufficient to extract even weak convergence of the pairing $\bar{\rho}_\tau \nabla \bar{\Pi}_\tau$. 
	
	Fortunately, as we shall show in Lemma \ref{lemma:convergence-rho-pi} below, thanks precisely to the horizontal regularity bound shown in Lemma \ref{lemma:bound-X-derivative-rho}, the $Y$-average of $\bar{\rho}_\tau$ against a smooth test function yields a sequence converging strongly to the corresponding average of the limiting $\rho$. Since the potential $\bar{\Pi}_\tau$ is merely a function on $X$, integration against this strongly convergent sequence of averages yields the desired convergence of the joint product. This intuition is made rigorous by the following Lemma:
	
	\begin{lemma}
		\label{lemma:convergence-rho-pi}
		There are functions $\rho \in C(\Rnn; \Gamma(\mu, \nu))$ and $\Pi \in L^2_{loc}(\Rnn; H^1(X))$ such that, in the limit $\tau \to 0$, at least along a suitable sequence and for each $T\in (0,\infty)$:
		\begin{align}
		\label{eq:convergence-rho}
		\bar{\rho}_\tau(t) &\weakto \rho(t) \quad \text{locally uniformly w.r.t.\ $t\in\Rnn$ in $\wass$},					
		\\
		\label{eq:convergence-pi}
		\bar{\Pi}_\tau &\weakto \Pi\quad \text{weakly in } L^2([0, T]; H^1(X)).
		\end{align}
		Moreover, define for any given $\zeta\in C^1_c(\Rnn\times X\times Y)$ the \emph{vertical averages}
		\begin{equation}
		\label{eq:def_omega}
		\omega_\tau(t,x) := \int_Y \zeta(t,x,y)\bar{\rhod}_\tau(t, x, y)\dd \nu(y)
		,
		\qquad
		\omega(t,x) := \int_Y \zeta(t,x,y)\rhod(t, x, y)\dd \nu(y).
		\end{equation}
		where $\bar{\rhod}_\tau \in L^1(X\times Y, \Lebesgue \otimes \nu)$ and $\rhod \in L^1(X\times Y, \Lebesgue \otimes \nu)$ denote respectively the densities of $\bar{\rho}_n$ and $\rho$ with respect to $\Lebesgue \otimes \nu$.
		Then, along the same sequence as above,
		\begin{align}
		\label{eq:convergence-rho-average}
		\omega_\tau&\to 
		\omega
		\quad \text{strongly in $L^2([0,T]\times X)$}, 	
		\\
		\label{eq:convergence-rho-pi-product}
		\omega_\tau\nabla\bar{\Pi}_\tau &\weakto 
		\omega\nabla\Pi
		\quad \text{weakly in $L^1([0,T]\times X)$}.  		
		\end{align}
		
		The limit ($\rho, \Pi)$ satisfies, for all $\psi \in C_c^\infty(\R_{+} \times X \times Y)$ with $\nabla_X \psi \cdot \mathbf{n} \equiv 0$ on $\partial X$ the weak equation:
		\begin{equation}
		\int_0^\infty
		\intXY
		\partial_t \psi \dd \rho
		=
		\int_0^\infty
		\intXY
		\big[
		\nabla_X \psi\cdot\nabla_X [V - \Pi] 
		-
		\kappa
		\Delta_X \psi
		\big]\dd \rho
		-
		\int_{X\times Y}
		\psi(0) \dd \rho^0.
		\label{eq:weaktxy}
		\end{equation}
		
		The potential $\Pi$ is given for a.e.~time as the unique solution satisfying $\int_X \Pi \dd \mu = 0$ of the elliptic equation: 
		\begin{equation}
		\label{eq:weak-potential}
		\int_X \nabla \xi \cdot \nabla \Pi(t)\,  \dd \mu
		=
		\int_X \nabla \xi\cdot  \left[
		\int_Y \nabla_X V(x,y) \rhod(t,x, y) \dd \nu(y)
		\right]
		\dd x
		-\kappa \int_X \Delta \xi \dd \mu
		\end{equation}
		for all $ \xi\in C^\infty(X)$ with $\nabla \xi \cdot \mathbf{n} = 0$ on $\partial X$.
	\end{lemma}
	\begin{proof}
		For the convergence of the trajectories \eqref{eq:convergence-rho} we employ the refined version of Ascoli-Arzela in \cite[Proposition 3.3.1]{SavareGradientFlows}: we use Lemma \ref{lem:equicontinuity-rho} for the equicontinuity of $\rho$ w.r.t.~$\wass$, while the pointwise compactness follows from the compactness of $(\prob(X\times Y), \wass)$ for $X$ and $Y$ compact. 
		The fact that $\nrg(\bar{\rho}_\tau)$ is uniformly bounded by $\nrg(\rho_0)$ imposes a uniform upper bound on $\entr(\bar{\rho}_\tau)$, which by lower-semicontinuity of the energy (Lemma \ref{lemma:energy-lsc}) implies that
		\begin{equation}
			\nrg(\rho(t)) 
			\le 
			\liminf_{\tau \to 0}
			\nrg(\bar{\rho}_\tau(t))
			\le 
			\nrg(\rho^0),
		\end{equation}
		and thus $\rho(t)$ must have a density with respect to $\Lebesgue \otimes \nu$ for all $t\ge0$, denoted by $\rhod(t)$.
		On the other hand, the weak convergence of the potentials \eqref{eq:convergence-pi} is a direct consequence of the uniform bound \eqref{eq:L2-bound-phi} and the weighted Poincaré inequality for functions with zero mean \cite{WeightedPoincare}.
		
		For \eqref{eq:convergence-rho-average} let us first identify the limit. For a fixed $\zeta \in C_c^1(\Rnn \times X\times Y)$, integrating with respect to a test function $\psi \in C_0(\Rnn \times X)$:
		\begin{align}
		\nonumber
		\int_0^\infty\intX
		&
		\psi(t,x)
		\omega_\tau(t,x) \dd x \dd t
		=\int_0^\infty \intXY
		\psi(t,x) \zeta(t,x, y)\dd \bar{\rho}_\tau(t,x, y) 
		\\
		\nonumber
		\to&
		\int_0^\infty \intXY
		\psi(t,x) \zeta(t, x, y)\dd \rho(t,x, y) 
		=
		\int_0^\infty \intX
		\psi(t,x)
		\left[
		\int_Y  \zeta(t, x, y)\rhod(t,x, y) \dd \nu(y)
		\right] \dd x \dd t
		\\
		&=
		\int_0^\infty\intX
		\psi(t,x)
		\omega(t,x) \dd x \dd t
		\label{eq:identify-limit-omega}
		\end{align}
		This shows at least weak* convergence of $\omega_\tau \dd x \dd t$ to $\omega_\tau \dd x \dd t$. For showing a stronger convergence, we use the version of the Aubin-Lions lemma given in \cite[Theorem 2]{RossiSavare}. As Banach space, we choose $B:=L^1(X)$, as properly coercive integrand, we use
		\[ \mathcal F(\omega) = \int_X |\nabla_X\omega|\dd x + \|\omega\|_{L^\infty(X)}, \]
		and the pseudo-distance is the dual norm in $C^1(X)$, i.e.,
		\begin{align*}
		g(\omega_1,\omega_2) = \sup\left\{\int_X \xi(x)\big(\omega_1(x)-\omega_2(x)\big)\dd x\, \middle|\, \xi\in C^1(X),\,\|\xi\|_{C^1(X)}\le 1\right\}.
		\end{align*}
		Let us first show that $\omega_\tau$ is uniformly bounded in $L^\infty([0,T]\times X)$: since $\rho_\tau(t)\in \Gamma(\mu, \nu)$ for a.e.~$t\in \Rnn$, it holds: 
		\begin{equation}
		|\omega_\tau(t,x)| 
		= 
		\left|
		\int_Y \zeta(t,x,y)\bar{\rhod}_\tau(t, x, y)\dd \nu(y)
		\right|
		\le 
		\| \zeta\|_\infty
		\int_Y\bar{\rhod}_\tau(t, x, y)\dd \nu(y)
		=
		\| \zeta\|_\infty
		\mu(x)   
		\label{eq:uniform-Linfty-bound}
		\end{equation}
		for a.e.~$(t, x)\in \Rnn\times X$, and the upper bound on $\mu$ completes the claim. 
		Further, $\omega_\tau$ inherits from $\bar{\rho}_\tau$ the regularity of its $X$-gradient: 
		\begin{align*}
		\int_0^T \left(\int_X |\nabla_X \omega_\tau| \dd x\right) \dd t
		&=
		\int_0^T \intX
		\left|
		\nabla_X
		\int_Y 
		\zeta(t,x,y) \bar{\rhod}_\tau(t,x,y)\dd \nu(y)
		\right|
		\dd t \dd x
		\\
		&\le 
		\int_0^T \intXY
		\left|
		\nabla_X(\zeta \bar{\rhod}_\tau)
		\right|
		\dd t \dd x
		\dd  \nu(y)
		\\
		&\le 
		T\|\nabla_X \zeta\|_\infty
		+		
		\|\zeta\|_\infty
		\int_0^T \intXY
		\left|
		\nabla_X\bar{\rhod}_\tau
		\right|
		\dd t \dd x
		\dd  \nu(y)
		\\
		&=
		T
		\|\nabla_X \zeta\|_\infty
		+
		\|\zeta\|_\infty
		\int_0^T \intXY
		2 \sqrt{\bar{\rhod}_\tau} \nabla_X (\sqrt{\bar{\rhod}_\tau})
		\dd t \dd x
		\dd  \nu(y)
		\\
		&\le 
		T
		\|\nabla_X \zeta\|_\infty
		+
		2\|\zeta\|_\infty \sqrt{T}
		\left(
		\int_0^T \intXY
		\big|\nabla_X\sqrt{\bar{\rhod}_\tau}\big|^2
		\dd t \dd x
		\dd  \nu(y)\right)^{1/2}.
		\end{align*}
		Where we used Cauchy-Schwarz in the last inequality.
		In virtue of Lemma \ref{lemma:bound-X-derivative-rho}, we conclude
		\begin{align*}
		\int_0^T \left(\int_X |\nabla_X \omega_\tau| \dd x\right) \dd t
		\le 
		T
		\|\nabla_X \zeta\|_\infty
		+
		\frac{\|\zeta\|_\infty}{2\kappa} \sqrt{T}(C + \nrg(\rho^0)).
		\end{align*}
		In combination with the $L^\infty$-bound, we thus have
		\begin{align}
		\label{eq:AL1}    
		\sup_\tau\int_0^T \mathcal F(\omega_\tau)\dd t < \infty.
		\end{align}
		Further, the $Y$-averages $\omega_\tau$ also inherit the uniform regularity in time from $\bar{\rho}_\tau$: testing with $\xi \in C^1(X)$, we obtain that
		\begin{align*}
		\int_{X}
		\xi(x)
		&
		[\omega_\tau(m\tau, x) - \omega_\tau(n\tau, x)] \dd x
		=
		\int_{X\times Y}
		\xi(x) [(\zeta\bar{\rhod}_\tau)(m\tau, x, y) - (\zeta\bar{\rhod}_\tau)(n\tau, x, y)] \dd x 
		\dd  \nu(y)
		\\
		&
		=
		\int_{X\times Y}
		\xi(x) [\zeta(m\tau, x, y) - \zeta(n\tau, x, y)] \bar{\rhod}_\tau(m\tau, x, y) \dd x 
		\dd  \nu(y)
		\\
		&\quad+
		\int_{X\times Y}
		\xi(x) \zeta(n\tau, x, y)[\bar{\rhod}_\tau(m\tau, x, y) -  \bar{\rhod}_\tau(n\tau, x, y)] \dd x 
		\dd  \nu(y)        
		\\
		&\le
		\|\xi\|_{C^1(X)} \|\zeta\|_{C^1(\Rnn\times X \times Y)} \big[
		\tau(m-n)
		+
		\wass_1 (\rho_\tau^m, \rho_\tau^n)
		\big]
		\\
		&\le		
		\|\xi\|_{C^1(X)} \|\zeta\|_{C^1(\Rnn\times X \times Y)} \big[
		\tau(m-n)
		+
		C\sqrt{\tau(m-n)}
		\big]
		\end{align*}
		where we used \eqref{eq:dual-wass1} in the first inequality, and Lemmas \ref{lemma:order-wass} and \ref{lem:equicontinuity-rho} in the last. In other words, we have that
		\begin{align}
		\label{eq:AL2}
		\int_0^{T-h} g\big(\omega_\tau(t+h),\omega_\tau(t)\big)\dd t \to 0
		\quad \text{as $h\searrow0$, uniformly in $\tau\in(0,1)$.}
		\end{align}
		Since \eqref{eq:AL1} and \eqref{eq:AL2} are verified, it follows by means of \cite[Theorem 2]{RossiSavare} that $\omega_\tau$ converges to a limit $\omega_*\in L^1([0,T]\times X)$, in $L^1(X)$, in measure with respect to time.
		Moreover, since the limit was already identified in \eqref{eq:identify-limit-omega}, in particular we have that $\omega_* = \omega$.
		 Invoking the uniform $L^\infty$-bound \eqref{eq:uniform-Linfty-bound} again, we conclude that actually $\omega\in L^\infty([0,T]\times X)$, and that $\omega_\tau\to\omega$ strongly in $L^p([0,T]\times X)$, for any $p\in[1,\infty)$. 
		
		In particular, the convergence $\omega_\tau\to\omega$ is strong in $L^2([0,T]\times X)$, and so \eqref{eq:convergence-rho-pi-product} follows by weak convergence of $(\nabla \bar{\Pi}_\tau)_\tau$.
		
		We now turn to the weak formulation. 
		Let $\psi \in C_c^\infty(\R_{+} \times X \times Y)$ with $\nabla_X \psi \cdot \mathbf{n} \equiv 0$ on $\partial X$.
		Since $\psi$ is compactly supported, for each $\tau$ there exists an $N$ such that $\psi \equiv 0$ for $t > N\tau$. Hence,
		\begin{align}
		\nonumber
		\int_0^\infty
		\intXY
		& \partial_t \psi(t, x,y) \dd \bar{\rho}_\tau(t,x,y)
		= 
		\sum_{n=0}^{N-1}
		\int_{n\tau}^{(n+1)\tau}
		\int_{X \times Y}
		\partial_t \psi(t, x,y) \dd  \rho_\tau^n(x,y)\dd t
		\\
		\nonumber
		&=
		\sum_{n=0}^{N-1}
		\int_{X \times Y}
		\big[
		\psi((n+1)\tau, x,y)
		-
		\psi(n\tau, x,y)
		\big]
		\dd \rho_\tau^n(x,y) 
		\\
		\nonumber
		&=
		-
		\int_{X \times Y}
		\psi(0, x,y) \dd \rho^0(x,y) 
		-
		\sum_{n=1}^{N-1}
		\tau
		\int_{X \times Y}
		\psi(n\tau, x,y) \dd \left[\frac{\rho_\tau^n - \rho_\tau^{n-1}}{\tau}\right]
		\end{align}	
		The first term is the last term of \eqref{eq:weaktxy}. For the second term we turn to \eqref{eq:weakxz}:
		\begin{align}
		\nonumber
		&\int_0^\infty
		\intXY
		 \partial_t \psi(t,x,y) \dd \bar{\rho}_\tau(t,x,y)
		+ 
		\int_{X \times Y}
		\psi(0, x,y) \dd \rho^0(x,y) 
		\\
		&=			
		-
		\sum_{n=1}^{N-1}
		\tau
		\left[
		-\int_{X\times Y} \nabla_X\psi(n\tau) \cdot\nabla_X\big[V-\Pi_\tau^n\big]
		\dd\rho_\tau^n
		+
		\kappa\int_{X\times Y} \Delta_X\psi(n\tau) \dd\rho_\tau^n
		+ 
		\tau
		R_\tau^n(\psi(n\tau))
		\right].
		\label{eq:continuity-eq-derivation}
		\end{align}
		Let us first control the remainder: 
		\begin{align*}
		\sum_{n=1}^{N-1}
		\tau^2
		R_\tau^n(\psi(n\tau))
		&\le
		\tau 
		\sum_{n=1}^{N-1}
		\tau
		R_\tau^n(\psi(n\tau))
		\le 
		\tau
		\|\psi\|_{C^2}
		\sum_{n=1}^{N-1}
		\frac{\wassF(\rho_\tau^n, \rho_\tau^{n-1})^2}{2\tau}
		\le
		\tau
		\|\psi\|_{C^2}
		\nrg(\rho^0).
		\end{align*}
		which yields an $O(\tau)$ contribution. Lemma \ref{lemma:convergence-rho-pi} allows to control the rest of integrands, so that replacing $\psi(n\tau)$ by $\psi(t)$ in \eqref{eq:continuity-eq-derivation} only incurs in an $o(1)$ error. Thus, back to \eqref{eq:continuity-eq-derivation}:
		
		\begin{align*}
		\int_0^\infty \intXY
		\partial_t \psi(t, x,y) \dd \bar{\rho}_\tau(t,x,y) 
		&=
		\int_0^\infty \intXY
		\nabla_X\psi(t, x,y)\cdot
		\nabla_X V(x,y)\dd \bar{\rho}_\tau(t,x,y)
		\\
		&\quad
		-
		\int_0^\infty \intX
		\nabla_X \bar{\Pi}_\tau(t,x)
		\cdot 
		\left[
		\int_Y\nabla_X \psi(t,x,y)
		\bar{\rhod}_\tau(t,x,y) \dd \nu(y)
		\right] \dd t \dd x
		\\
		&\quad
		-
		\int_0^\infty \intXY
		\Delta_X\psi(t,x,y)
		\dd\bar{\rho}_\tau(t,x,y)
		\\
		&\quad
		-
		\int_{X \times Y}
		\psi(0,x,y) \dd \rho^0(x,y) 
		+
		o(1).
		\end{align*}
		The first and third lines converge by sheer integration of the weak* convergent $(\bar{\rho}_\tau)_\tau$ against continuous functions, while the second line converges in virtue of \eqref{eq:convergence-rho-pi-product}, and the last line is constant (up to the $o(1)$ term).
		
		Finally, \eqref{eq:weak-potential} follows from testing \eqref{eq:weaktxy} with $\psi(t,x,y) = \theta(t)\xi(x)$, for $\theta$ with compact support and satisfying $\theta(0) = 0$:
		\begin{align*}
		0
		&=
		\int_0^\infty \intX
		\partial_t \theta(t) \xi(x) \dd\mu(x)	\dd t
		=
		\int_0^\infty \intXY
		\partial_t \theta(t) \xi(x) \dd \rho(t,x,y)	
		\\
		&=
		\int_0^\infty \intXY
		\theta(t)
		\nabla \xi(x)\cdot\nabla_X [V - \Pi]\dd \rho(t,x,y)
		-
		\kappa
		\int_0^\infty \intXY
		\theta(t)
		\Delta_X \xi(x) \dd\rho(t,x,y).
		\\
		&=
		\int_0^\infty \intX
		\theta(t)
		\nabla \xi(x)\cdot
		\left[
		-   \mu\nabla \Pi	
		+
		\int_Y \nabla_X V(\cdot, y) \rhod(\cdot, \cdot, y)
		\dd \nu(y)
		\right]
		-
		\kappa
		\int_0^\infty \intX
		\theta(t)
		\Delta_X \xi(x) \dd\mu(x) \dd t,
		\end{align*}
		and the claim follows from the arbitrariness of $\theta$.
	\end{proof}
	
		\begin{lemma}
		\label{lemma:dissipation0}
		Let $(\rho_t, \Pi_t)_t$ be a limit trajectory as in Lemma \ref{lemma:convergence-rho-pi} and $\rhod_t\in L^1(X\times Y, \Lebesgue \otimes \nu)$ denote the density of $\rho_t$ with respect to $\Lebesgue \otimes \nu$.
		Define the  dissipation of the energy $\nrg$ as: 
		\begin{equation}
		\mathcal{I}(\rho_t, \Pi_t)
		:=
		\int_{X\times Y}
		|\nabla_X (V - \Pi_t + \kappa \log \rhod_t)|^2 \dd \rho_t.
		\end{equation}
		Then for all $0 \le t_1 < t_2 \le \infty$ the integral of the dissipation is lower-semicontinuous along the converging subsequences in Lemma \ref{lemma:convergence-rho-pi}, i.e.:
		\begin{equation}
		\int_{t_1}^{t_2} \mathcal{I}(\rho_t, \Pi_t) \dd t 
		\le 
		\liminf_{k\to\infty}
		\int_{t_1}^{t_2} \mathcal{I}((\bar{\rho}_{n})_t, ( \bar{\Pi}_n)_t) \dd t.		
		\end{equation}
		In particular,
		\begin{equation}
		0 
		\le 
		\nrg(\rho_{t_1})
		+
		\int_{t_1}^{t_2} \mathcal{I}(\rho_t, \Pi_t) \dd t 
		\le 
		\nrg(\rho^0).
		\label{eq:integral-dissipation}
		\end{equation}
	\end{lemma}

\begin{proof}
	For simplicity in this proof we will remove the bars from $\bar{\rho}^n$ and $\bar{\Pi}^n$; simply note that they refer to the piecewise constant  discrete trajectories. Then:
	
	\begin{align}
		\nonumber
		\int_{t_1}^{t_2}&\mathcal{I}((\rho_n)_t, (\Pi_n)_t) \dd t
		=
		\int_{t_1}^{t_2}\int_{X\times Y}
		|\nabla_X (V - \Pi_n 
		+\kappa \log \rhod_n)|^2  \dd\rho_n 
		\\
		\label{eq:dissipation-1}
		&=
		\int_{t_1}^{t_2}\int_{X\times Y}
		|\nabla_X (V + \kappa \log \rhod_n)|^2 \dd \rho_n
		\\
		\label{eq:dissipation-2}
		&\quad
		-
		\int_{t_1}^{t_2}\int_{X\times Y}
		2
		\nabla_X (V + \kappa \log \rhod_n) \cdot \nabla_X \Pi_n \dd\rho_n 
		\\
		\label{eq:dissipation-3}
		&\quad
		+
		\int_{t_1}^{t_2}\int_{X}
		|\nabla \Pi_n|^2 \dd t \dd \mu.
	\end{align}
	For \eqref{eq:dissipation-1}, defining $\sigma_n := \rhod_n e^{V/\kappa}$ we have
	\begin{align*}
		\int_{t_1}^{t_2}\int_{X\times Y}
		|\nabla_X (V + \kappa \log \rhod_n)|^2 \dd \rho_n
		&=
		\kappa^2
		\int_{t_1}^{t_2}	\int_{X\times Y}
		|\nabla_X \sigma_n|^2 \sigma_n e^{-V/k} \dd x \dd \nu(y)
		\\
		&=
		4\kappa^2
		\int_{t_1}^{t_2}\int_{X\times Y}
		e^{-V/k}
		|\nabla_X \sqrt{\sigma_n}|^2\dd x \dd \nu(y).
	\end{align*}
	which is lower-semicontinuous, since $\sigma_n$ inherits $\rho_n$'s regularity properties outlined in Lemma \ref{lemma:regularity-weak-solutions}. The next term \eqref{eq:dissipation-2} can be rewritten as follows: 
	\begin{align*}
		\int_{t_1}^{t_2}\int_{X\times Y}
		&
		\nabla_X (V + \kappa \log \rhod_n) \cdot \nabla_X \Pi_n \dd \rho_n 
		=
		\\
		&=
		\int_{t_1}^{t_2}\int_{X} \nabla \Pi_n \cdot \left[
		\int_Y \rho_n \nabla_X V \dd \nu(y)
		\right] \dd t\dd x
		+
		\kappa 
		\int_{t_1}^{t_2}\int_{X} \nabla \Pi_n \cdot
		\nabla \log \mu \dd \mu,
	\end{align*}
	where the first contribution converges thanks to Lemma \ref{lemma:stability-weak-solutions} and the second by sheer weak convergence of $\nabla (\Pi_n)_t$ (plus the fact that $\log\mu\in H^1(X,\mu)$, thanks to the finite Fisher information of $\mu$). The last term \eqref{eq:dissipation-3} is lower-semicontinuous under weak convergence of $(\Pi_n)_n$ in $L^2([t_1, t_2]; H^1(X, \mu))$. 
	
	Finally, \eqref{eq:integral-dissipation} is a direct consequence of \eqref{eq:dissipation} and the lower-semicontinuity of both energy and the integrated dissipation.
\end{proof}

	\begin{lemma}
		\label{lemma:regularity-weak-solutions}
		The weak solutions $(\rho, \Pi)$ of \eqref{eq:weaktxy}-\eqref{eq:weak-potential}
		extracted as the limit of the minimizing movement scheme in Lemma \ref{lemma:convergence-rho-pi} satisfy, for all $0<t_1 <t_2 < \infty$:
		\begin{align}
		\label{eq:equicontinuity-limit}
		\wass^2(\rho_{t_1}, \rho_{t_2})
		\le 
		\wassF^2(\rho_{t_1}, \rho_{t_2})
		&\le 
		\nrg(\rho_0)(t_2 - t_1).
		\\
		\label{eq:bound-pi-limit}
		\int_{t_1}^{t_2}
		\int_X \big|\nabla\Pi_t\big|^2 \dd t\dd \mu(x) 
		&\le
		C(t_2 - t_1).
		\\
		\label{eq:bound-rho-limit}
		2\kappa^2\int_{t_1}^{t_2} \int_{X\times Y}\big|\nabla_X\sqrt{\rhod_t}\big|^2\dd t\dd x \dd \nu(y)
		&\le 
		\nrg(\rho^0) + C (t_2 - t_1).
		\end{align}		
		with $C$ a finite constant depending only on $\|\nabla_X V\|_\infty$, $\|\nabla \sqrt\mu\|_{2}$ and $\kappa$.
		
	\end{lemma}
	
	\begin{remark}
		Even though the weak solutions constructed in Lemma \ref{lemma:convergence-rho-pi} can (a priori) only be shown  to be continuous in the $\wass$ metric, a posteriori Lemma \ref{lemma:regularity-weak-solutions} yields regularity in time with respect to the $\wassF$ metric.
	\end{remark}
	
	\begin{proof}[Proof of Lemma \ref{lemma:regularity-weak-solutions}]
		For \eqref{eq:equicontinuity-limit} we will use the Benamou-Brenier formula fiberwise. For given $t_1 < t_2$, and for each $s \in [0,1]$ consider $\varrho_s \assign  \rho((1-s)t_1 + s t_2)$ with density given by $\varrhod_s\in L^1(X\times Y, \Lebesgue \otimes \nu)$  and $\pi_s \assign \Pi((1-s)t_1 + s t_2)$. For each $y\in Y$, define the velocity field: 
		\begin{equation}
			v_s : (x,y) \mapsto (t_2 - t_1)\nabla_X [
			V(x,y) - \pi_s(x) + \kappa \log \varrhod_s(x,y)
			],
		\end{equation} 
		
		The pair $(\varrhod_s, v_s)$ thus constructed is a feasible weak solution for the  continuity equation:
		\begin{equation}
		\begin{gathered}
			\partial_s \varrhod_s(x, y) + \nabla_X \cdot (\varrhod_s(x, y) v_s(x,y)) = 0\\
			\nabla_X (\varrhod_s v_s)= 0
			\quad
			\text{on }\partial X
			\end{gathered}
		\end{equation}
		
		Now, by the Benamou-Brenier formula,
		\begin{align*}
			\wassF(\varrho_0, \varrho_1)^2
			&=
			\int_Y \wass (\varrhod_0(\cdot, y), \varrhod_0(\cdot, y))
			\dd \nu(y)
			\le 
			\int_Y
			\left[
			\int_0^1 
			\int_X
			|v_s(x,y)|^2 \varrhod_s(x,y)  \dd x\dd s
			\right] \dd \nu(y)
			\\
			&=
			\int_0^1 
			\intXY
			|v_s(x,y)|^2 
			\dd \varrho_s(x,y) \dd s
			\\
			&=
			(t_2 - t_1)^2
			\int_0^1 
			\intXY
			|\nabla_X [
			V(x,y) - \pi_s(x) 
			+ \kappa \log \varrhod_s(x,y)
			]|^2 
			\dd \varrho_s(x,y) \dd s,
		\end{align*}
		and after a change of variables we eventually obtain:
		\begin{align*}
		\wassF(\rho(t_1), \rho(t_2))^2
		&=
		(t_2 - t_1)
		\int_{t_1}^{t_2}
		\intXY
		|\nabla_X [
		V(x,y) - \Pi_t(x) + \kappa \log \rhod_t(x,y)
		]|^2 
		\dd \rho_t(x,y) \dd t	
		\\
		&
		\le 
		(t_2 - t_1) \nrg(\rho^0),
		\end{align*}
		where in the last inequality we used Lemma \ref{lemma:dissipation0}. 
				
		For \eqref{eq:bound-pi-limit} we test \eqref{eq:weak-potential} with (an approximating sequence of) $\Pi_t$ and integrate by parts to obtain:
		\begin{align*}
		\intX &|\nabla \Pi_t|^2 \dd \mu(x)
		=
		\intXY \nabla \Pi_t \cdot \nabla_X V \dd \rho_t
		+
		\kappa 
		\intX \nabla \Pi_t \cdot \nabla \log \mu \dd \mu(x)
		\\
		&\le 
		\|\nabla_X V\|_\infty
		\left(
		\intXY |\nabla \Pi_t|^2 \dd \rho_t
		\right)^{1/2}
		+
		\kappa 
		\left(
		\intXY |\nabla \Pi_t|^2 \dd \mu
		\right)^{1/2}
		\left(
		\intXY |\nabla \log\mu|^2 \dd \mu
		\right)^{1/2},
		\end{align*}
		where we used Cauchy-Schwartz. Then, using that $\|\nabla \log \mu \|_{L^2(X, \mu)} =2\|\nabla \sqrt\mu \|_{L^2(X)}$ and simplifying one $\|\nabla\Pi_t\|_{L^2(X,\mu)}$ term yields:
		\begin{equation}
		\label{eq:final-bound-pit}
		\left(
		\intXY |\nabla \Pi_t|^2 \dd \mu
		\right)^{1/2}
		\le 
		\|\nabla_X V\|_\infty
		+
		2\kappa 
		\|\nabla \sqrt\mu \|_{2},
		\end{equation}
		and \eqref{eq:bound-pi-limit} follows by integration in time.
		
		To obtain \eqref{eq:bound-rho-limit} we expand \eqref{eq:integral-dissipation}:
		\begin{align}
		\label{eq:expand-dissipation-1}
		\nrg(\rho_0) 
		&\ge 
		\int_{t_1}^{t_2}
		\mathcal{I}(\rho_t, \Pi_t) \dd t
		=
		\int_{t_1}^{t_2}
		\int_{X\times Y}
		|\nabla_X (V - \Pi + \kappa \log \rhod)|^2 \dd \rho	
		\\
		\label{eq:expand-dissipation-2}
		&=
		\int_{t_1}^{t_2}
		\int_{X\times Y}
		|\nabla_X V|^2 \dd \rho	
		+
		\int_{t_1}^{t_2}
		\int_{X}
		|\nabla\Pi|^2 \dd \mu \dd t	
		+
		\kappa^2
		\int_{t_1}^{t_2} \intXY
		|\nabla_X \log \rhod|^2 \dd \rho
		\\
		&
		\label{eq:expand-dissipation-3}
		\quad
		+2
		\int_{t_1}^{t_2}
		\int_{X\times Y}		
		\left[
		\kappa
		\nabla_X V\cdot \nabla_X \log \rhod  
		-
		\kappa
		\nabla_X \log\rhod\cdot \nabla_X \Pi  
		-
		\nabla_X V\cdot \nabla_X \Pi 
		\right]
		\dd \rho
		\end{align}
		To bound the Fisher information of $\rho$ we solve for the last term in \eqref{eq:expand-dissipation-2}:
		\begin{align*}
		4\kappa^2
		\int_{t_1}^{t_2} \intXY
		|\nabla_X \sqrt{\rhod}|^2 \dd x\dd \nu(y)
		&\le
		\nrg(\rho^0)
		-2
		\int_{t_1}^{t_2}
		\int_{X\times Y}
		\kappa
		\nabla_X V\cdot \nabla_X \log \rhod 
		\dd \rho
		\\
		&\quad 
		+
		2
		\int_{t_1}^{t_2}
		\int_{X\times Y}
		\left[
		\kappa
		\nabla_X \log\rhod\cdot \nabla_X \Pi  
		-
		\nabla_X V\cdot \nabla_X \Pi 
		\right]
		\dd \rho
		\end{align*}
		
		Now all the inner products are straightforward to bound:
		\begin{align*}
		-2\kappa
		\int_{t_1}^{t_2}
		\int_{X\times Y}&
		\nabla_X V\cdot \nabla_X \log \rhod
		\dd \rho
		\\
		&\le 
		2\kappa
		\int_{t_1}^{t_2}
		\left(
		\int_{X\times Y}
		|\nabla_X V|^2
		\dd \rho_t
		\right)^{1/2}
		\left(
		\int_{X\times Y}
		|\nabla_X \log \rhod|^2  
		\dd \rho_t
		\right)^{1/2}
		\dd t
		\\
		&\le
		4\kappa 
		\|\nabla_X V \|_\infty
		\sqrt{t_2 - t_1}
		\left(
		\int_{t_1}^{t_2}
		\int_{X\times Y}
		|\nabla_X \sqrt{\rhod}|^2  
		\dd x \dd \nu(y)
		\dd t
		\right)^{1/2}
		\\
		&\le 
		2\|\nabla_X V \|_\infty^2 (t_2 - t_1)
		+
		2\kappa^2 
		\int_{t_1}^{t_2}
		\int_{X\times Y}
		|\nabla_X \sqrt{\rho}|^2 
		\dd x \dd \nu(y) \dd t,
		\end{align*}
		where we used $ab \le a^2/2 + b^2/2$ for $a = 2\|\nabla_X V \|_\infty
		\sqrt{t_2 - t_1}$ and $b$ encompassing the rest. For the next term:
		
		\begin{align*}
		2\kappa\int_{t_1}^{t_2}
		\int_{X\times Y}
		\nabla_X \log\rhod\cdot \nabla_X \Pi  
		\dd \rho
		&\le
		2\kappa
		\int_{t_1}^{t_2}
		\intX
		\nabla \mu\cdot \nabla \Pi 
		\dd x
		\le 
		2\kappa
		\int_{t_1}^{t_2}
		\intX
		\nabla \log \mu \cdot \nabla \Pi 
		\dd \mu
		\\
		&\le 
		4\kappa
		\left(
		\int_{t_1}^{t_2}
		\intX
		|\nabla \Pi|^2
		\dd \mu
		\right)^{1/2}
		\sqrt{t_2 - t_1}
		\left(
		\intX
		|\nabla\sqrt{\mu}|^2
		\dd x
		\right)^{1/2}
		\\
		&\le 
		2\kappa
		\int_{t_1}^{t_2}
		\intX
		|\nabla \Pi|^2
		\dd \mu
		+
		2\kappa
		(t_2 - t_1)
		\intX
		|\nabla\sqrt{\mu}|^2
		\dd x,
		\end{align*}
		where we used a similar bound. For the remaining term:
		\begin{align*}
		2\int_{t_1}^{t_2}
		\int_{X\times Y}
		\nabla_X V\cdot \nabla_X \Pi  
		\dd \rho
		&\le	
		2
		\|\nabla_X V \|_\infty
		\sqrt{t_2 - t_1}
		\left(
		\int_{t_1}^{t_2}
		\intX
		|\nabla \Pi|^2
		\dd \mu
		\right)^{1/2}
		\\
		&\le 
		\|\nabla_X V \|_\infty^2
		(t_2-t_1)
		+
		\int_{t_1}^{t_2}
		\intX
		|\nabla \Pi|^2
		\dd \mu		
		\end{align*}
		
		Finally, collecting all the contributions and using \eqref{eq:bound-pi-limit} yields the desired bound.
	\end{proof}

	\section{Stability of weak solutions}
	\label{sec:stability}
	
	In the following sections, for a weak solution $(\rho, \Pi)$ of \eqref{eq:weaktxy}-\eqref{eq:weak-potential} we will denote by $\rho_t$ the measure $\rho(t)$ (which is for a.e.~$t$ a feasible coupling in $\Gamma(\mu, \nu)$ with density with respect to $\Lebesgue\otimes \nu$ given by $r_t\in L^1(X\times Y , \Lebesgue \otimes \nu)$), and analogously we will denote by $\Pi_t$ the function $\Pi(t)$.
	
	We start with a preliminary result that will be useful below: 
	\begin{lemma}[Stability of solutions]
		\label{lemma:stability-weak-solutions}
		Let $(\mu_n)_n\subset \prob(X)$ be a sequence with uniformly bounded Fisher information $\|\nabla \sqrt{\mu_n}\|_2 \le C < \infty$ and uniformly bounded from above and away from zero, converging strongly to $\mu$ in $L^2(X)$, and let $(\nu_n)_n\subset\prob(Y)$ converge to $\nu$ weak*. For each $n\in \mathbb{N}$ let $(\rho_n, \Pi_n)$ be a weak solution to the continuity equation \eqref{eq:weaktxy}-\eqref{eq:weak-potential}
		with initial condition $\rho_n^0 \in \Gamma(\mu_n, \nu_n)$ and satisfying the estimates in Lemma \ref{lemma:regularity-weak-solutions}.
		Further assume that $(\rho_n^0)_n$ converges weak* to $\rho^0\in \Gamma(\mu, \nu)$, and that $\nrg(\rho^0_n \mid \mu_n, \nu_n)$ is uniformly bounded in $n$. 
		
		Then, up to the extraction of a subsequence, $(\rho_n)_n$ converges on compact sets to a weak solution of \eqref{eq:weaktxy}-\eqref{eq:weak-potential} with initial condition $\rho^0$. More precisely, for each $T>0$,
		\begin{align}
		\label{eq:convergence-rho-limit}
		\rho_n &\weakto \rho\quad \text{locally uniformly w.r.t.~$t\in \Rnn$ in $\wass$,}\\
		\label{eq:convergence-pi-limit}
		\Pi_n &\weakto \Pi\quad \text{weakly in } L^2([0,T]; H^1(X)).
		\end{align}	
		
		Moreover, define for any given $\zeta\in C^1_c(\Rnn\times X \times Y)$:
		\begin{equation}
		\label{eq:def_omega-limit}
		\omega_n(t,x) := \int_Y \zeta(t, x,y)\rhod_n(t, x, y)\dd \nu(y)
		,
		\qquad
		\omega(t,x) := \int_Y \zeta(t, x,y)\rhod(t, x, y)\dd \nu(y).
		\end{equation}
		where $\rhod_n \in L^1(X\times Y, \Lebesgue \otimes \nu_n)$ and $\rhod \in L^1(X\times Y, \Lebesgue \otimes \nu)$ denote respectively the densities of $\rho_n$ and $\rho$.
		Then, along the same sequence as above,
		\begin{align}
		\label{eq:convergence-rho-average-limit}
		\omega_n&\to 
		\omega
		\quad \text{strongly in }L^2([0,T]\times X),
		\\
		\label{eq:convergence-rho-pi-product-limit}
		\omega_n\nabla\Pi_n &\weakto 
		\omega\nabla\Pi
		\quad 
		\text{weakly in }L^1([0,T]\times X).
		\end{align}
	\end{lemma}
	
	\begin{proof}
		As in Lemma \ref{lemma:convergence-rho-pi}, for the convergence of the trajectories \eqref{eq:convergence-rho-limit} we employ the refined version of Ascoli-Arzela in \cite[Proposition 3.3.1]{SavareGradientFlows}: we use Lemma \ref{lemma:regularity-weak-solutions} for the equicontinuity of $\rho$ w.r.t.~$\wass$, while the pointwise compactness follows from the compactness of $(\prob(X\times Y), \wass)$ for $X$ and $Y$ compact. This yields existence of a $\wass$-continuous limit trajectory $\rho$, to which $(\rho_n)_n$ converges locally uniformly. Besides, $\rho_t$ has marginals $\mu$ and $\nu$ for all $t\ge 0$ by weak* convergence of $(\mu_n, \nu_n)$ to $(\mu, \nu)$ and continuity of the projection operators $\proj_X$ and $\proj_Y$.
		
		To show density of $\rho_t$ with respect to $\Lebesgue \otimes \nu$ note first that, by Lemma \ref{lemma:dissipation0}, $\nrg(\rho_n(t) \mid \mu_n, \nu_n) \le \nrg(\rho^0_n \mid \mu_n, \nu_n) \le C$, since by assumption the energy of the initial datum is uniformly bounded, and by Lemma \ref{eq:integral-dissipation} the energy is non-increasing.
		On the other hand, the joint lower-semicontinuity of the energy with respect to weak* convergence (Lemma \ref{lemma:energy-lsc}) implies:
		\begin{equation}
			\nrg(\rho(t) \mid \mu, \nu) 
			\le 
			\liminf_{n\to\infty} \nrg(\rho_n(t)\mid \mu_n, \nu_n)
			\le 
			\liminf_{n\to\infty} \nrg(\rho_n^0\mid \mu_n, \nu_n)
			\le C,
		\end{equation}
		which shows that $\rho(t)$ has a density with respect to $\Lebesgue \otimes \nu$ for all $t \ge 0$.
		
		Now, the bounds \eqref{eq:bound-pi-limit} and \eqref{eq:bound-rho-limit} are uniform in $n$, since they only depend on the energy of the initial datum, the Fisher information of $\mu_n$, $\kappa$ and $\|\nabla_X V\|_\infty$, all of which are uniformly bounded by assumption. In view of this, \eqref{eq:bound-pi-limit} together with the the weighted Poincaré inequality for functions with zero mean \cite{WeightedPoincare} imply the weak convergence of the potentials $\Pi_n$ to a limit trajectory $\Pi \in L^2([0, T]; H^1(X,\mu))$ for all $T > 0$. Further, the uniform control on the horizontal Fisher information of $(\rho^n)_n$ given by \eqref{eq:bound-rho-limit} implies convergence of the vertical averages \eqref{eq:convergence-rho-average-limit} and \eqref{eq:convergence-rho-pi-product-limit} in a similar fashion to the proof in Lemma \ref{lemma:convergence-rho-pi}. The main difference is that now the marginals $\mu_n$ and $\nu_n$ vary along the converging sequence, but this can be factored into the proof without major issues. 	
		
		The remaining step is to show that $\rho$ and $\Pi$ verify their respective continuity equations. Let $\psi \in C_c^\infty(\R_{+} \times X \times Y)$ with $\nabla_X \psi \cdot \mathbf{n} \equiv 0$ on $\partial X$, and test \eqref{eq:weakx} with $\psi$ for the weak solution $(\rho_n, \Pi_n)$:
		\begin{align}
		\int_{X\times Y}
		\psi(0) \dd \rho^0_n
		&=
		\int_{\Rnn\times X \times Y}
		\left[
		\nabla_X \psi\cdot\nabla_X V 
		- \nabla_X \psi\cdot\nabla\Pi^k 
		-
		\kappa
		\Delta_X \psi 
		-
		\partial_t \psi\right]
		\dd \rho^n.
		\label{eq:weak-formulation-stability}
		\end{align}
		The left-hand side converges by weak* convergence of the initial conditions. Likewise, the first, third and last term of the right-hand side converge by sheer weak* convergence of $\rho^n$ to $\rho$. The remaining term can be shown to converge in virtue of \eqref{eq:convergence-rho-pi-product-limit}: defining $(\mathbf{u}^k)_n$ and $\mathbf{u}$ as  
		\begin{equation}
		\mathbf{u}_n(t,x) := \int_Y \nabla_X\psi(t,x,y)\rhod_n(t, x, y)\dd \nu(y)
		,
		\qquad
		\mathbf{u}(t,x) := \int_Y \nabla_X \psi(t,x,y)\rhod(t, x, y)\dd \nu(y),
		\end{equation}
		and using \eqref{eq:convergence-rho-pi-product-limit} we obtain convergence also of the remaining term from \eqref{eq:weak-formulation-stability}:
		\begin{align*}
		\int_0^T
		\int_{ X \times Y}
		&
		\nabla_X \psi\cdot\nabla\Pi_n \dd \rho_n
		=
		\int_0^T
		\intX \nabla \Pi_n(x) \cdot  \mathbf{u}_n(t,x) \dd t \dd x
		\\
		&\to 
		\int_0^T
		\intX
		\nabla \Pi(x) \cdot  \mathbf{u}(t,x) \dd t \dd x
		=
		\int_0^T
		\int_{ X \times Y}
		\nabla_X \psi\cdot\nabla\Pi \dd\rho.
		\end{align*}
		where $T$ the threshold after which $\psi(t)$ is identically zero.
		
		Finally, the fact that $\Pi$ obeys its respective weak equation \eqref{eq:weak-potential} follows by testing \eqref{eq:weak-formulation-stability} with a test function with no dependence in $y$, plus noting that $\rho_n$ converges weak* to $\rho$ and that $\mu_n$ converges strongly to $\mu$ in $L^2(X)$:
		\begin{align*}
		\int_0^T\int_X \nabla_X \xi \cdot \nabla \Pi\, \dd t\dd \mu 
		&=
		\lim_n\int_0^T\int_X \nabla_X \xi \cdot \nabla \Pi_n\,\dd t \dd \mu_n
		\\
		&=
		\lim_n
		- 
		\kappa \int_0^T\int_X \Delta_X \xi\dd t
		\dd \mu_n
		+
		\int_0^T\int_{X\times Y}\nabla_X \xi
		\cdot 
		\nabla_X V  \dd \rho_n
		\\
		&=
		- 
		\kappa \int_0^T\int_X \Delta_X \xi\dd t
		\dd \mu
		+
		\int_0^T\int_{X\times Y}\nabla_X \xi
		\cdot 
		\nabla_X V  \dd \rho
		\end{align*}
	\end{proof}

	\begin{remark}
		Lemma \ref{lemma:stability-weak-solutions} allows to extract the continuum limit of finite, multiphasic flows such as those studied in~\cite{CGM17} by crafting a pertinent sequence of discrete second marginals $\nu_n$ converging to a weak* limit $\nu$. We illustrate this convergence with numerical examples in Section \ref{sec:numerics}.
	\end{remark}

	\section{Asymptotic limit of weak solutions}
	\label{sec:asymptotic}

 In this section we show that any weak solution $(\rho_t, \Pi_t)$ converges as $t\to\infty$ to the unique minimizer of the energy $\nrg$.
 This contrasts with the behavior of the solutions of the AHT scheme (corresponding to $\kappa = 0$ and $\rho_t$ being of Monge type, cf.~Section \ref{sec:motivation} and also Section \ref{sec:unregularized} further below), which are not guaranteed to converge to the optimizer and can stay stationary in suboptimal configurations.
	
	\begin{lemma}
		\label{lemma:dissipation}
		For a sequence of weak solutions $(\rho_n, \Pi_n)$ with initial conditions $(\rho^0_n)_n$ and fixed marginals $\mu$ and $\nu$ converging to $(\rho, \Pi)$ on a time interval $(t_1, t_2) \subset [0, \infty)$ under the same conditions of Lemma \ref{lemma:stability-weak-solutions}, the integral of the dissipation is lower-semicontinuous, i.e.:
		\begin{equation}
		\int_{t_1}^{t_2} \mathcal{I}(\rho_t, \Pi_t) \dd t 
		\le 
		\liminf_{k\to\infty}
		\int_{t_1}^{t_2} \mathcal{I}((\rho_{n})_t, (\Pi_n)_t) \dd t.		
		\end{equation}
	\end{lemma}

	\begin{proof}
		The proof is analogue to that of Lemma \ref{lemma:dissipation0}, replacing the regularity estimates coming from the minimizing movement scheme by the weak solution estimates in Lemma \ref{lemma:regularity-weak-solutions}.
	\end{proof}

	\begin{lemma}
		\label{lemma:dissipation-0-means-optimal}
		Let $\rho_\infty \in \Gamma(\mu, \nu)$, with a non-negative density $\rhod_\infty \ge 0$ in $L^1(X\times Y,\Lebesgue \otimes \nu)$, and let  $\Pi_\infty$ be the corresponding pressure given by \eqref{eq:weak-potential}. If $\mathcal{I}(\rho_\infty, \Pi_\infty) = 0$, then $\rho_\infty$ is the unique minimizer of the energy $\nrg$.
	\end{lemma}
	\begin{proof}
		By assumption:
		\begin{equation}
		\label{eq:dissipation-limit}
		0
		=
		\int_{X\times Y} 
		\left|\nabla_X\big(\kappa\log\rhod_\infty(x,y)+V(x,y)-\Pi_\infty(x)\big)\right|^2
		\dd \rho_\infty(x,y).
		\end{equation}
		
		First we will show that $\Pi_\infty$ is bounded. Since $\rhod_\infty$ is non-negative, the gradient in \eqref{eq:dissipation-limit} must vanish wherever $\rhod_\infty$ is strictly positive, i.e.
		\begin{equation}
		\nabla \Pi_\infty(x) 
		=
		\nabla_X V(x,y) + \kappa \nabla_X \log \rhod_\infty(x,y)
		\quad 
		\text{for $\rho_\infty$-a.e.~}(x,y).
		\end{equation}
		Multiplying by $\rho_\infty$ and integrating in $Y$ one finds:
		\begin{equation}
		\nabla \Pi_\infty(x) \mu(x)
		=
		\int_Y \nabla_X V(x,y) \rhod_\infty(x,y) \dd \nu(y) 
		+ 
		\kappa \nabla \mu(x)
		\quad 
		\text{for a.e.~}x,
		\end{equation}
		or
		\begin{equation}
		\nabla \Pi_\infty(x) 
		=
		\int_Y \nabla_X V(x,y) \frac{\rhod_\infty(x,y)}{\mu(x)} \dd\nu(y)
		+ 
		\kappa \nabla \log \mu(x)
		\quad 
		\text{for a.e.~}x.
		\end{equation}
		Since $\int_Y \rhod_\infty(x,y) \dd \nu(y) = \mu(x)$ and $\nabla_X V$ is bounded, we deduce $\Pi_\infty - \kappa \log \mu$ is bounded on $X$, and in view of the uniform bounds on $\mu$ so is $\Pi_\infty$. Now, defining $\sigma := e^{(V-\Pi_\infty)/\kappa} \rhod_\infty$ we can rewrite \eqref{eq:dissipation-limit} as:
		\begin{equation}
		0
		=
		\int_{X\times Y} 
		\left|\kappa\nabla_X\log\sigma \right|^2\sigma e^{-(V-\Pi_\infty)/\kappa}
		\dd x\dd \nu
		=
		4\kappa^2\int_{X\times Y} 
		\left|\nabla_X\sqrt{\sigma} \right|^2  e^{-(V-\Pi_\infty)/\kappa}
		\dd x\dd \nu.
		\end{equation}
		And since both $V$ and $\Pi_\infty$ are bounded, $\kappa$ is positive and $X$ has connected interior, $\sigma(x,y)$ must be for each $y$ equal to a positive constant $e^{\Psi(y)/\kappa}$. 
		Rearranging terms yields:
		
		\begin{equation}
		\label{eq:diagonal-scaling}
		\rhod_\infty(x,y)
		=
		\exp \left(
		\frac{\Pi_\infty(x) + \Psi(y) - V(x,y)}{\kappa}
		\right).
		\end{equation}
		
		Finally, since the measure $\rho_\infty = \rhod_\infty \dd x \dd \nu (y)$ has marginals $\mu$ and $\nu$ and its density takes the form of the diagonal scaling \eqref{eq:diagonal-scaling}, we conclude by Lemma \ref{lemma:entropic-minimizers} that it is the unique minimizer of $\nrg$, or analogously, the optimal entropic plan between $\mu$ and $\nu $ with cost function $V$ and regularization strength $\kappa$. The optimal entropic dual potentials are thus $\Pi_\infty$ and $\Psi$. 
	\end{proof}
	
	\begin{lemma}
		The weak solutions $(\rho_t, \Pi_t)_t$ obtained in Lemma \ref{lemma:convergence-rho-pi} tend as $t\to \infty$ to the unique minimizer of $\nrg$.
	\end{lemma}
	\begin{proof}
		Our convergence argument follows a similar rationale to that in \cite{asymptotic-convergence-reference}.
		Let $(\rho, \Pi)$ be a weak solution of \eqref{eq:weaktxy}-\eqref{eq:weak-potential}.
		There exist cluster points of the trajectory $(\rho_t)_t$ by sheer weak* compactness in the space of probability measures on $X\times Y$; let us denote such a cluster point by $\rho_*$ and the converging subsequence of times by $(t_n)_n$. By weak* continuity, $\rho_*$ preserves the marginals and positivity of the trajectory $\rho$.

		Let us define functions $\rho_n(s) := \rho(t_n + s)$, $\Pi_n := \Pi(t_n + s)$ for $s\in [0,1]$. Since each $(\rho_n, \Pi_n)$ is a weak solution of \eqref{eq:weaktxy}-\eqref{eq:weak-potential} with initial condition $\rho_n(0) = \rho(t_n)  \to \rho_*$, the stability of weak solutions (Lemma \ref{lemma:stability-weak-solutions}) grants convergence of (a subsequence of) $(\rho_n, \Pi_n)_n$ to a weak solution $(\rho', \Pi')$ with initial condition $\rho_*$. Now, since the dissipation $\mathcal{I}$ of the original weak solution $(\rho, \Pi)$ is positive and summable, the integrals of the form $\int_0^1 \mathcal{I}(\rho_n(s), \Pi_n(s)) \dd s$ must converge to zero, which combined with Lemma \ref{lemma:dissipation} implies that the dissipation vanishes in the asymptotic limit: 
		\begin{equation}
		0 \le \int_{0}^1 \mathcal{I}(\rho'(s), \Pi'(s)) \dd s
		\le 
		\liminf
		\int_{0}^1 \mathcal{I}(\rho_n(s), \Pi_n(s)) \dd s
		= 
		0
		\end{equation}
		Since $\mathcal{I}$ is a non-negative quantity (for $\rho \ge 0$), this means that $\mathcal{I}(\rho'(s), \Pi'(s))$ is identically zero, which by Lemma \ref{lemma:dissipation-0-means-optimal} implies that $\rho'(s)$ and $\rho_*$ are equal to the unique minimizer of $\nrg$, given by \eqref{eq:diagonal-scaling}. Since this convergence holds along any subsequence, the whole trajectory converges to $\rho_*$.
	\end{proof}
	
	\section{Stationarity in the unregularized case}
	\label{sec:unregularized}

    Section \ref{sec:asymptotic} showed that solutions to \eqref{eq:PDE1}-\eqref{eq:PDE2} converge as $t\to\infty$ to the minimizer of $\nrg$.
    In contrast, the AHT scheme does not enjoy such a global convergence. 
    As explained in Section \ref{sec:motivation}, equations \eqref{eq:PDE1}-\eqref{eq:PDE2} introduce two kinds of modifications with respect to the AHT scheme: a relaxation (enlarging the space of transport maps to that of transport plans) and a smoothing (adding an entropic regularization term to the energy).
    In this section we show that the relaxation step is not sufficient to endow the AHT scheme with convergence properties analogous to those of \eqref{eq:PDE1}-\eqref{eq:PDE2}, and that some kind of regularization appears to be necessary.
 
	As reviewed in Section \ref{sec:motivation}, the problem of existence of solutions becomes in general much harder for $\kappa = 0$, since the a priori estimate on the horizontal gradients of $\rho$ (or $\rhod$) is lost. 
	In fact, for an absolutely continuous $\nu$, weak solutions to \eqref{eq:PDE1}-\eqref{eq:PDE2} may be extremely degenerate, possibly giving mass to a single point in each $y$-fiber. Due to these difficulties we will not attempt to provide a comprehensive picture of the $\kappa = 0$ case. Instead, we will exemplify that, when weak solutions exist, they face the same type of obstacle for asymptotic convergence to the energy minimizer that is observed in the AHT scheme. 
	The reason is that in the $\kappa = 0$ regime there exist suboptimal stationary points of the evolution equation. We will demonstrate this by showing that, whenever $V\in C^2(X\times Y)$ satisfies the \textit{twist condition} (i.e., $\nabla_X V(x, \cdot)$ is injective for all $x$), any initial datum $\rho_0$ of the form 
	\begin{equation}
	\label{eq:form-stationary-points}
	\rho_0 = (\id, T_*)_\sharp \mu, 
	\qquad 
	\tn{with $T_*(x) = \nabla_X V(x, \cdot)^{-1}(\nabla\Pi_*(x))$, 
		\qquad
		$T_{*\sharp}\mu = \nu$}
	\end{equation}
	with $\Pi_*$ in $C^2(X)$ stays stationary by the discrete minimizing movement when $\tau$ is small enough. 
	As a consequence, in this case a limit trajectory exists and is constantly equal to the initialization, which in general will be suboptimal for the energy.
	
	Let us show this in detail. 
	For $\kappa = 0$ the minimizing movement \eqref{eq:JKO-scheme} becomes:
	\begin{equation}
	\label{eq:JKO-kappa-0}
	\rho_\tau^n
	\in \argmin_{\rho\in \Gamma(\mu, \nu)}
	\frac1{2\tau}\wassF(\rho, \rho_\tau^{n-1})^2 + \nrg_0(\rho),
	\qquad
	\nrg_0(\rho) := \int_{X\times Y} V \dd \rho.
	\end{equation}

	Our argument hinges on a primal-dual argument for the minimizing movement. In analogy to Lemma \ref{lem:steppositive}, \eqref{eq:JKO-kappa-0} can be rewritten as 
	\begin{equation}
	\label{eq:primal-kappa-0}
	\rho^n = \proj_{13}\gamma^n, 
	\quad
	\text{with $\gamma^n$ }
	\in 
	\argmin_{\gamma \in \Gamma(\mu, \rho^{n-1})}
	\int_{X\times X \times Y}
	\left(
	\frac{|x-x'|^2}{2\tau} + V(x,y) 
	\right)
	\dd \gamma(x,x',y),     
	\end{equation}
	with respective (pre)-dual problem
	\begin{align}
	\label{eq:dual-kappa-0}
	\begin{gathered}
	\sup_
	{\scriptsize\begin{array}{c}
		\Pi\in C(X)\\ \Psi\in C(X\times Y)
		\end{array}
	}
	\int_X \Pi(x) \dd \mu(x)
	+
	\int_{X\times Y}
	\Psi(x',y) \dd \rho^{n-1}(x',y)
	\\
	\text{s.t. }
	\Pi(x) + \Psi(x',y)
	\le 
	\frac{|x-x'|^2}{2\tau} + V(x,y)
	\text{ for all } (x, x', y) \in X\times X \times Y.
	\end{gathered}
	\end{align}
	
	We can show that $\rho_0$ is stationary by the minimizing movement ---or, equivalently, optimal for the primal problem \eqref{eq:primal-kappa-0} when choosing $\rho^{n-1}=\rho_0$ --- by building a dual certificate. Choosing $\Pi = \Pi_*$, we can obtain a feasible $\Psi$ by computing the $c$-transform of $\Pi_*$ with respect to the cost function $c(x,x',y) := \tfrac{|x-x'|^2}{2\tau} + V(x,y)$:
	\begin{align*}
	\Psi_*(x', y) 
	:=
	\inf_{x\in X}
	\frac{|x-x'|^2}{2\tau} + V(x,y)
	-
	\Pi_*(x).
	\end{align*}
	For $\tau$ sufficiently small, and for each $(x', y) \in X \times Y$, the expression in the infimum is a uniformly convex function of $x$ (thanks to the assumed smoothness of $\Pi_*$).
    So a sufficient condition for $\hat{x}$ to attain the infimum is that it satisfies:
    \begin{equation}
	(\hat{x}-x')/\tau + \nabla_X V(\hat{x},y) - \nabla_* \Pi(\hat{x}) = 0.
	\label{eq:optimality-condition-dual}
	\end{equation}
	
	Now, for $(x',y)$ in the support of $\rho_0=(\id, T_*)_\sharp \mu$ , \eqref{eq:form-stationary-points}, one has $y = \nabla_X V(x', \cdot)^{-1}(\nabla \Pi_*(x'))$, so choosing $\hat{x} = x'$ verifies \eqref{eq:optimality-condition-dual}. Thus, for $\rho^{n-1}$-a.e.~$(x',y)$, $\Psi_*$ is given by: 
	\begin{equation}
	\Psi_*(x', y) = V(x',y) - \Pi_*(x').
	\end{equation}
	
	The dual score then reads:
	\begin{equation}
	\int_X \Pi_* \dd \mu
	+
	\int_{X\times Y}
	\Psi_* \dd \rho^{n-1}
	=
	\int_X \Pi_* \dd \mu
	+
	\int_{X\times Y}
	(V - \Pi_*) \dd \rho^{n-1}
	=
	\int_{X\times Y}
	V \dd \rho^{n-1},
	\end{equation}
	which matches the primal score \eqref{eq:primal-kappa-0} for $\gamma_* := (\id, \id, T)_\sharp \mu$. Hence we obtained a primal-dual certificate, showing that for $\tau$ sufficiently small $\rho_0$ is a minimizer of \eqref{eq:primal-kappa-0} and therefore a stationary point of the minimizing movement scheme \eqref{eq:JKO-kappa-0}.
	
	\section{Numerical simulations}
	\label{sec:numerics}
	
	In this Section we showcase with numerical simulations the main theoretical results discussed above. 
	Our numerical experiments are based on solving the minimizing movement scheme \eqref{eq:JKO-scheme}, or more precisely its formulation as a convex optimization problem \eqref{eq:primal}, using the Chambolle-Pock primal-dual method \cite{Chambolle2011}. 
    Our approach differs from traditional numerical approaches to simulate multiphase porous media flows (see for instance~\cite{CE97, AS79, EHM03}) as it builds on the variational interpretation of the problem rather than on a direct discretization of the partial differential equations~\eqref{eq:PDE1}--\eqref{eq:PDE2}. 
    Let us mention the contribution~\cite{CGLM19} where an augmented Lagrangian approach extending the ideas of~\cite{BCL16} to the context of multiphase porous media flows  was compared to the more classical (and non-variational) phase-by-phase upstream mobility finite volume scheme~\cite{EHM03}.

	To keep the figures intuitive and simple we fix $X = Y = [0,1]$, $V(x,y) = |x-y|^2/2$, and $\mu = \dd x$ the Lebesgue measure on $X$. We set $\kappa = 0.01$.
	Unless otherwise stated we discretize $\mu$ to a set of $M = 256$ equispaced points with equal weights in $[0,1]$. We set the second marginal $\nu$ likewise to a sum of $N\in \{4, 16, 64, 256\}$ equispaced dirac deltas with equal weights. For the initialization $\rho^0$ we choose either the \textit{product} initialization $\rho^0 = \mu \otimes \nu$ or the \textit{flipped} initialization $\rho^0 = (\id, 1 - \id)_\sharp \mu$, each discretized appropriately. 
	
	We fix the timestep to $\tau = 0.25$. For smaller values of $\tau$ we encounter discretization artefacts that hinder the convergence of the JKO step:
    When $X$ is discretized, the minimal cost for moving a particle horizontally becomes quantized, since it must move by at least one pixel, and it may therefore be more expensive than the potential gain by following a small $X$-gradient of $V$, especially when $\tau$ is very small.
	To alleviate this issue one could use entropic regularization in the transport term $\wassF$  (see e.g.~\cite{entropic-schemes}); however in this work we prefer to illustrate faithfully the behavior of the minimizing movement without introducing additional terms.
	
	Figure \ref{fig:short-trayectory-product} shows an example of the first iterations in the minimizing scheme for a product initialization and $N = 4$.
	Note how the accumulation of mass of species 1 at the left boundary and of species 4 at the right boundary makes up for the gap left by the other species. 
	
		\begin{figure}[htbp]
		\centering
		\includegraphics[width=\textwidth]{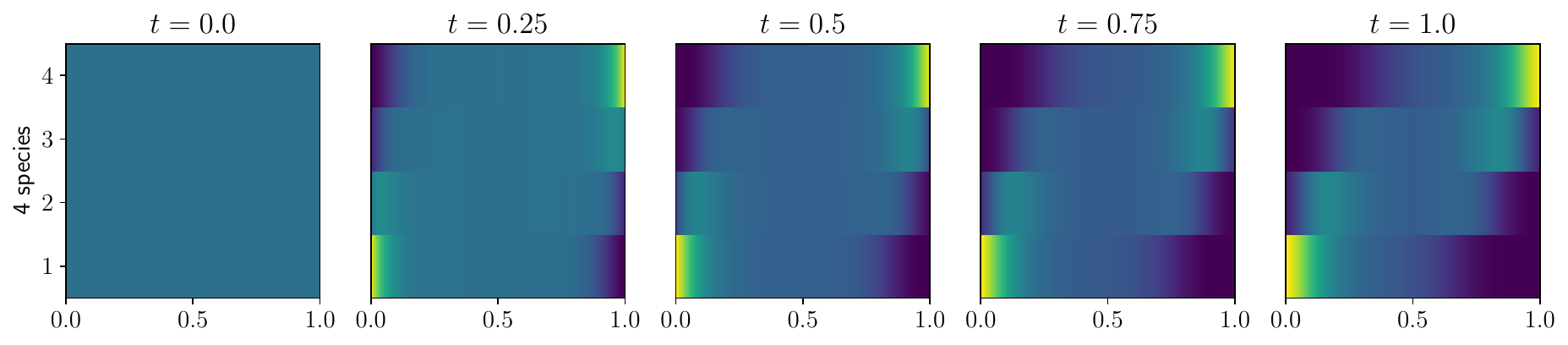}
		\caption{Discrete trajectory $\bar{\rho}_\tau$ for $N = 4$ and product initialization. Each band represents a phase, with hue indicating density. }
		\label{fig:short-trayectory-product}
	\end{figure}
	
	\begin{figure}
		\centering
		\includegraphics[width=\linewidth]{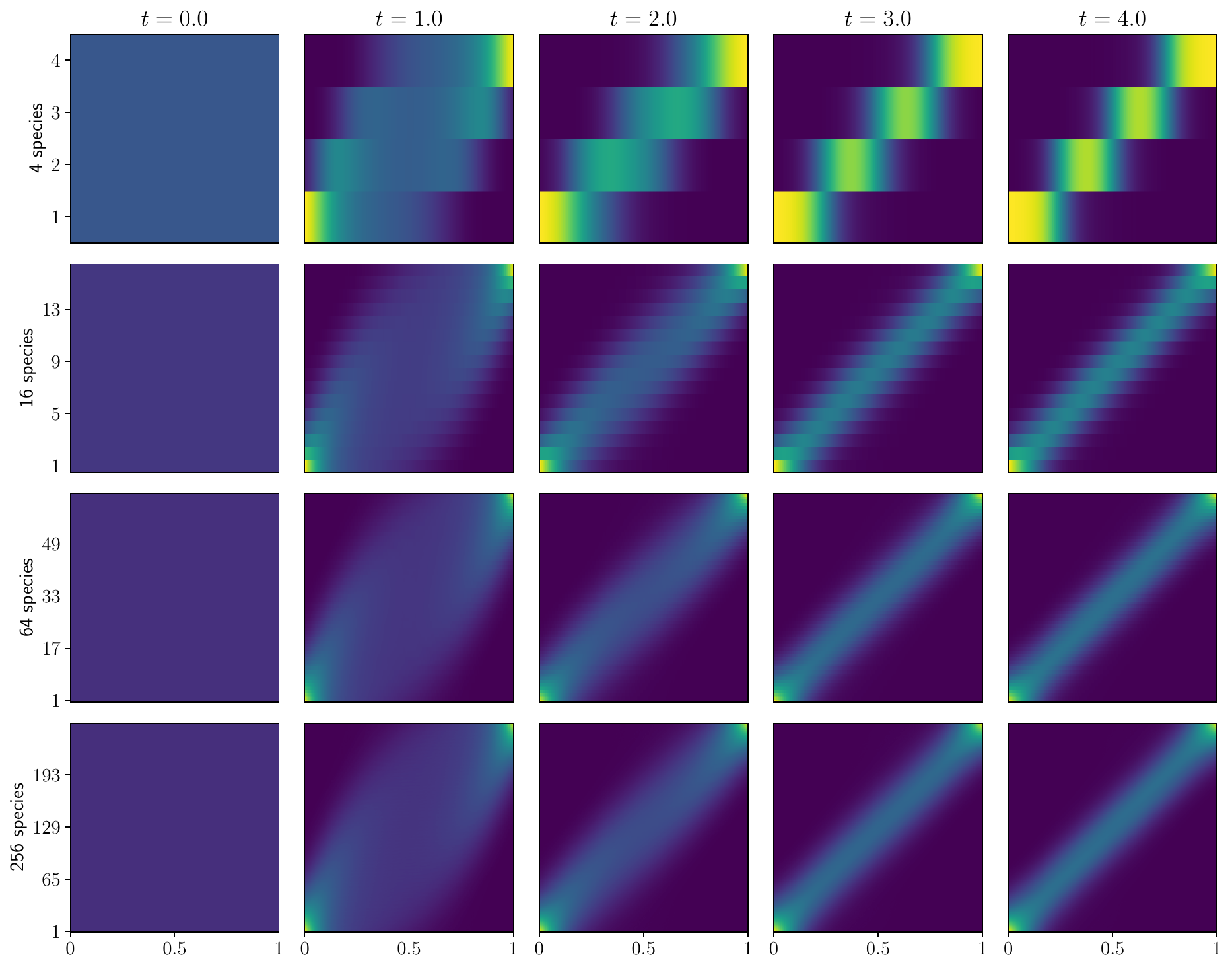}
		\caption{Comparison of discrete trajectories for different number of species $N$ and product initialization. Note how the trajectories tend to a consistent limit as $N$ increases.}
		\label{fig:compare-trayectory-product}
	\end{figure}

	Figure \ref{fig:compare-trayectory-product} shows a comparison of the same evolution for several choices of the number of species $N$. It is interesting to note how the evolution is consistent for the several values of $N$ (i.e.~convergence is achieved in approximately the same time), and the intermediate steps in the evolution are strikingly similar, in agreement with the stability of weak solutions shown in Lemma \ref{lemma:stability-weak-solutions}.

	An analogous observation can be made for Figures \ref{fig:short-trayectory-flipped} and \ref{fig:compare-trayectory-flipped}. Here the initialization is a flipped plan, which constitutes a worse initial score for $\nrg$ than the product coupling from Figures \ref{fig:short-trayectory-product} and \ref{fig:compare-trayectory-product}.
	This explains the longer time to convergence; however the final coupling remains essentially identical to the final coupling for the product initialization. Note that as $N\to\infty$ this example tends to a trajectory which does not fulfill our assumptions since the initial entropy $\entr(\rho^0)$ tends to $\infty$ as $N \to \infty$. Nevertheless, the convergence to the global optimizer still appears to hold.

	\begin{figure}[htbp]
		\centering
		\includegraphics[width=\linewidth]{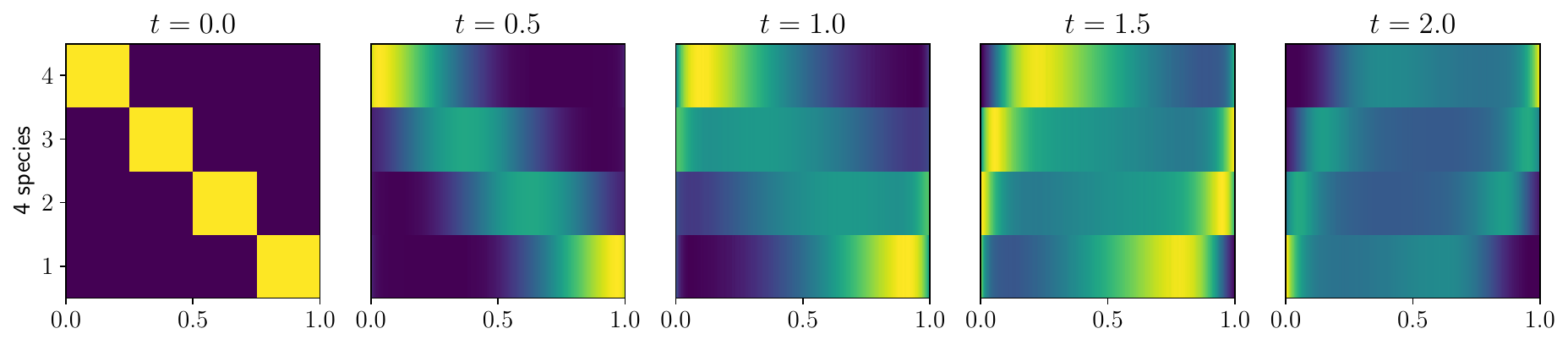}
		\caption{Discrete trajectory $\bar{\rho}_\tau$ for $N = 4$ and flipped initialization.}
		\label{fig:short-trayectory-flipped}
	\end{figure}
	
	\begin{figure}
		\centering
		\includegraphics[width=\linewidth]{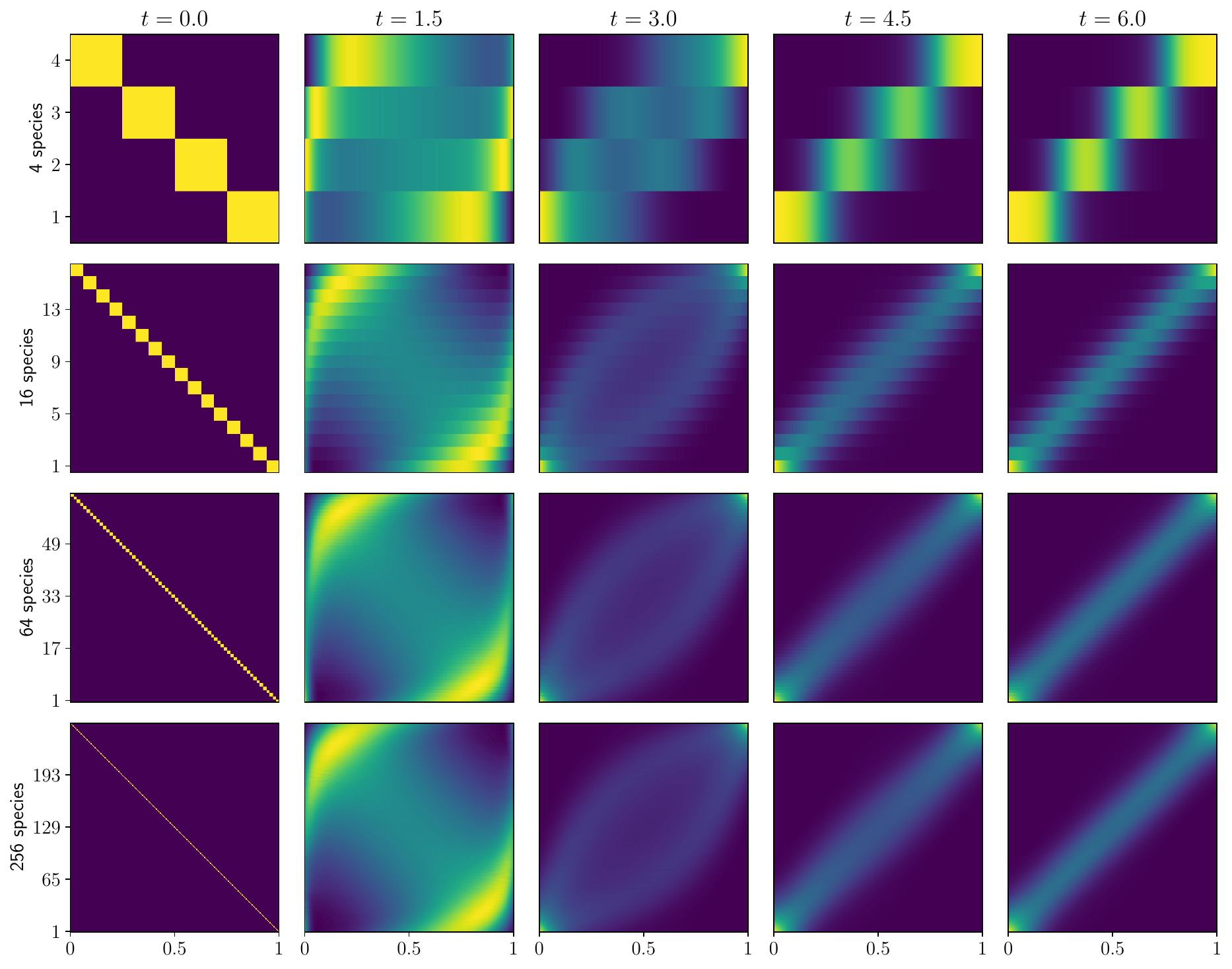}
		\caption{Comparison of discrete trajectories for different number of species $N$ and flipped initialization. Note how the trajectories again tend to a consistent limit, even when the initial conditions have unbounded entropy as $N \to \infty$}
		\label{fig:compare-trayectory-flipped}
	\end{figure}

	Convergence to the minimizer of $\nrg$ is examined in greater detail in Figures \ref{fig:compare-convergence} and \ref{fig:convergence-rate}. Figure \ref{fig:compare-convergence} shows the final discrete iterate (corresponding to $t=10$) under the two studied initializations and several choices of $N$ with the corresponding optimizer of $\nrg$, namely $\rho_*$, obtained with the Sinkhorn algorithm. There appar to be no discernable differences in the images. For a quantitative analysis, in Figure \ref{fig:convergence-rate} we plot the relative score difference
	\begin{equation}
	\Delta \nrg(\rho_t) 
	\assign 
	\frac{\nrg(\rho_t) - \nrg(\rho_*)}{\nrg(\rho_*)}
	\ge 
	0
	\end{equation}
	with respect to the elapsed time $t$, for the two initializations and either $M = N = 64$ or $M = N = 256$. There are two features of interest in the curves. The first is that $\Delta \nrg$ decreases until it reaches a plateau, which is encountered first for lower resolutions. This is due to discretization artefacts, as mentioned at the beginning of this section.
    The second interesting feature is that, even though the flipped initialization takes a longer time to converge (as it has a higher initial score), the final score is approximately independent on the initialization. Furthermore, the convergence rate seems to be independent of the initialization and approximately linear, indicated by the straight, parallel lines Figure \ref{fig:convergence-rate}.
	
		\begin{figure}
		\centering
		\includegraphics[width=0.8\linewidth]{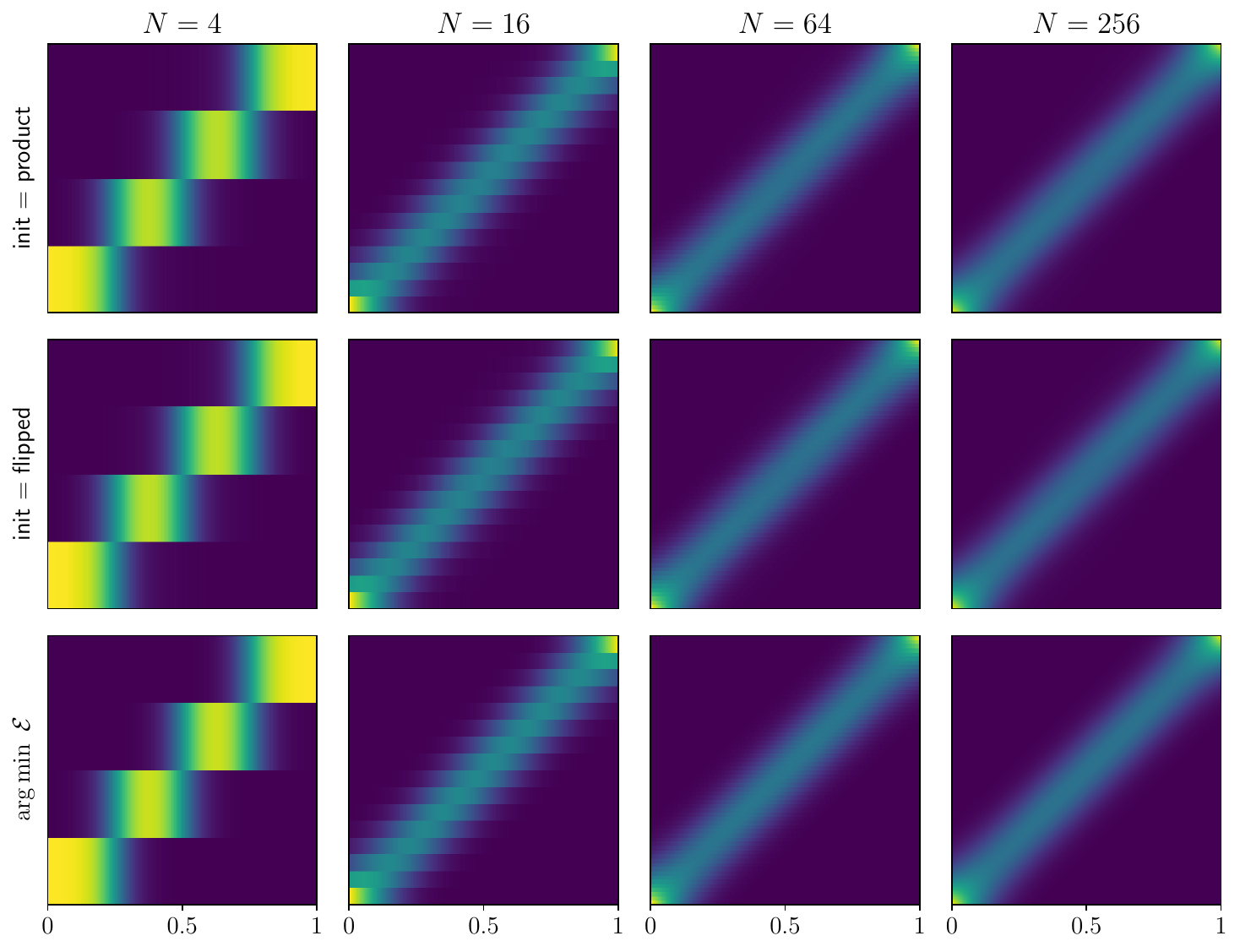}
		\caption{Comparison of the final discrete iterate (corresponding to $t=10$) and the minimizer of $\nrg$.}
		\label{fig:compare-convergence}
	\end{figure}
	
	\begin{figure}
		\centering
		\includegraphics[width=0.5\linewidth]{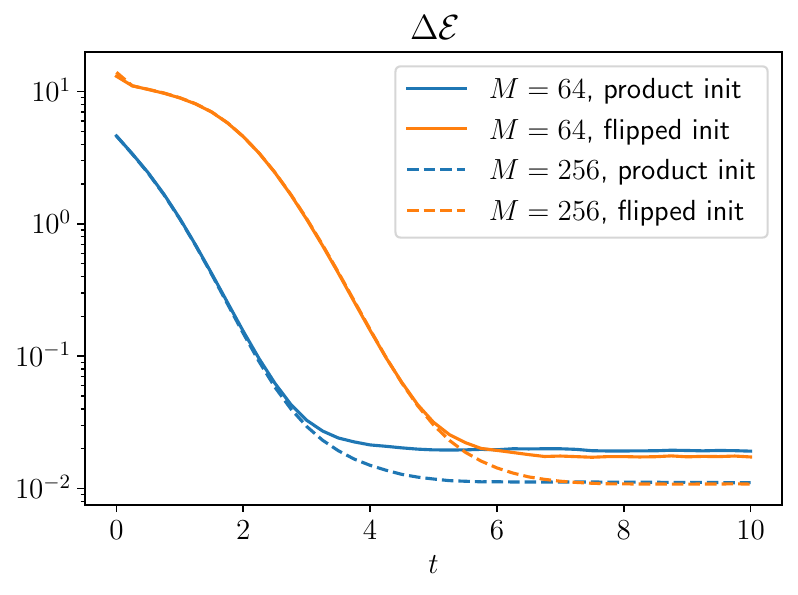}
		\caption{Evolution of the relative difference between $\nrg(\rho_t)$ and $\nrg(\rho_*)$.}
		\label{fig:convergence-rate}
	\end{figure}

	\begin{figure}
		\centering
		\includegraphics[width=\linewidth]{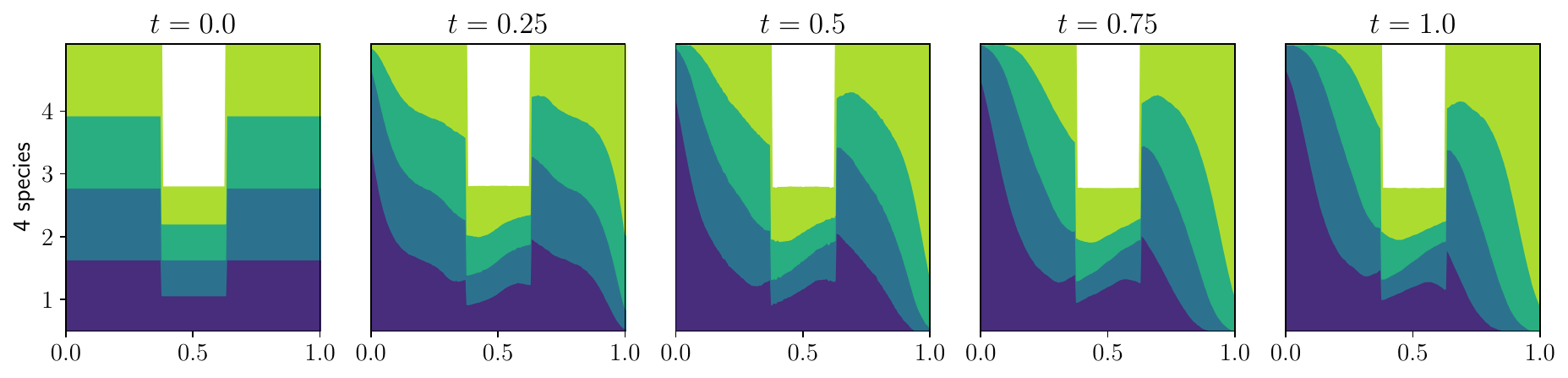}
		\caption{Discrete trajectory $\bar{\rho}_\tau$ for $\mu$ featuring a bottleneck, $N = 4$ and product initialization.
		We change the visualization with respect to previous figures to showcase better the effect of the bottleneck on the density of the different phases. Here the width the different bands represents the density of the respective species.
		Their combined densities add up to $\mu$, which is given by \eqref{eq:density-bottleneck}.
	}
		\label{fig:short-trayectory-bottleneck}
	\end{figure}
	
	\begin{figure}
		\centering
		\includegraphics[width=\linewidth]{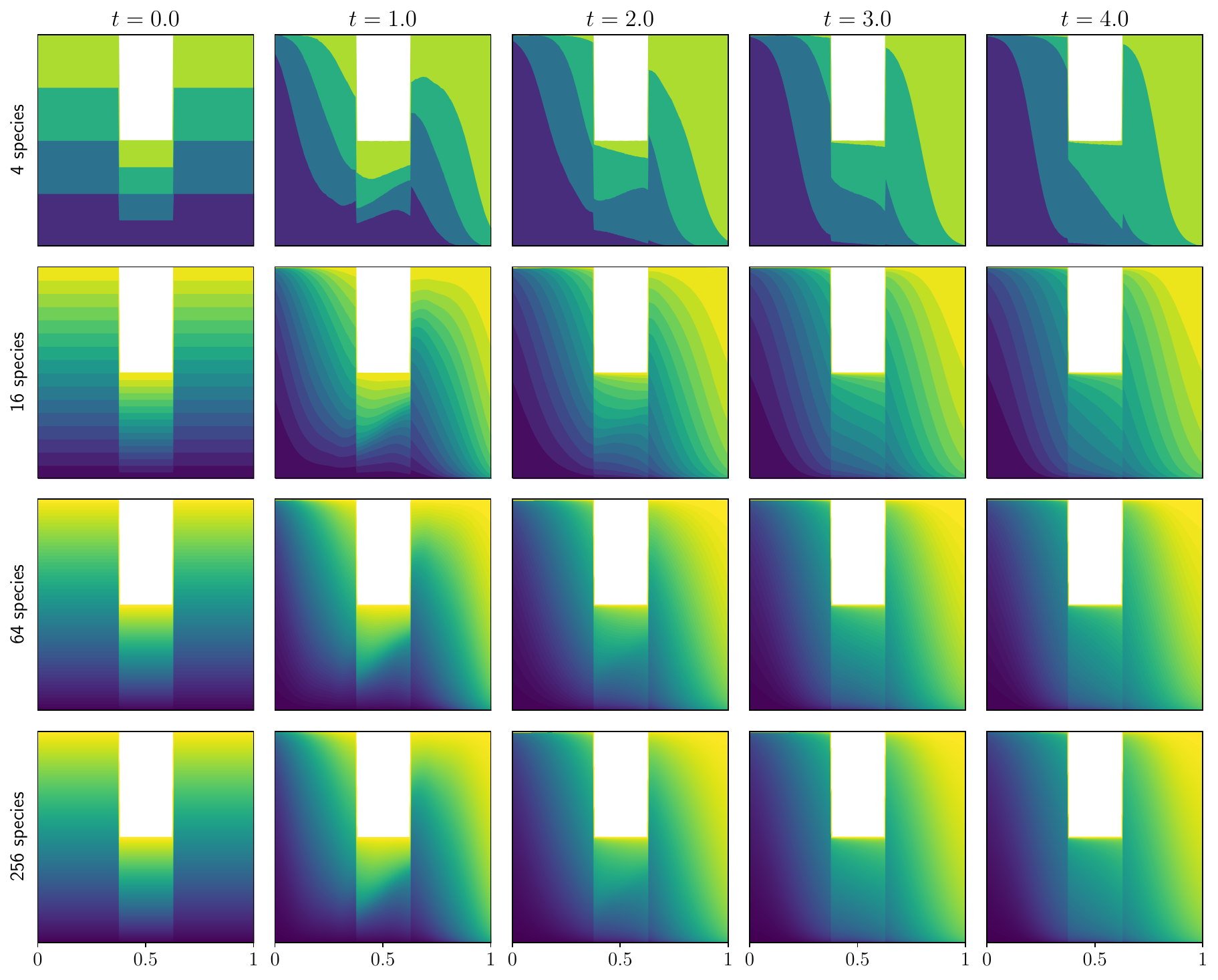}
		\caption{Comparison of discrete trajectories for $\mu$ featuring a bottleneck, product initialization and different number of species $N$. Note how the trajectories again tend to a consistent limit, even when $\mu$ has (in the continuum limit) unbounded Fisher information.}
		\label{fig:compare-trayectory-bottleneck}
	\end{figure}
	
	Finally, Figures \ref{fig:short-trayectory-bottleneck} and \ref{fig:compare-trayectory-bottleneck} show a trajectory for a $\mu$ with discontinuous density, hence not satisfying our Fisher information bound assumption.
    Concretely, we use the density
    \begin{equation}
        \mu(x) \assign \begin{cases} a & \tn{for } x \in [0,1] \setminus [0.5-\delta,0.5+\delta], \\
        b & \tn{for } x \in [0.5-\delta,0.5+\delta].
        \end{cases}
        \label{eq:density-bottleneck}
    \end{equation}
	with $\delta = 1/8$, $b = a/2$ and $a$ chosen so that $\|\mu\| = 1$
    This creates a bottleneck with reduced capacity for mass transport in the middle region.
	The discontinuity in $\mu$ is reproduced by each slice $r(\cdot, y)$ and we observe congestion at the boundaries of the bottleneck region; however the relative density $r(\cdot, y) / \mu(\cdot)$ remains continuous for all times.
	Again, the trajectories show a consistent behavior for all values of $N$ considered, which may indicate that the Fisher information bound hypothesis could be further relaxed. 
	
	In conclusion, the numerical simulations of this section exemplify the stability of the weak solutions of \eqref{eq:weaktxy}-\eqref{eq:weak-potential} as the second marginal $\nu$ tends to a continuum, as well as the asymptotic convergence to the minimizer of the energy $\nrg$ as time goes to infinitiy. Moreover, there is some numerical evidence that the convergence rate may be linear.
    The simulations provide interesting insights on how a continuum of particles (the species $y \in Y$) find an optimal arrangement of positions (the points $x \in X$) by local exchanges, subject to a volume constraint, that may include obstructions and bottlenecks.

\appendix

\section{A technical lemma}
\begin{lemma}
    \label{lem:HB}
    Let $\gamma\in\prob(X\times Y)$ with respective marginals $\mu\in\prob(X)$ and $\nu\in\prob(Y)$ be given, and assume that $\gamma$ possesses a density $u\in L^1(X\times Y, \mu\otimes\nu)$ with respect to $\mu\otimes\nu$. Further, let $f\in L^1(X\times Y, \gamma)$ have the following property: for every $\theta\in L^\infty(X\times Y, \mu \otimes \nu)$ satisfying
    \begin{equation}        
        \label{eq:app-adm}
        \begin{split}
        \int_{X}f(x,y)\,\theta(x,y) u(x,y)\dd\mu(x) &= 0 \quad\text{for $\nu$-a.e.~$y\in Y$}, \\
        \int_{Y}f(x,y)\,\theta(x,y) u(x,y)\dd\nu(y) &= 0 \quad\text{for $\mu$-a.e.~$x\in X$},
        \end{split}
    \end{equation}
    it follows that
    \begin{align}
        \label{eq:app-vanish}
        \int_{X\times Y} f(x,y)\,\theta(x,y)\dd\gamma(x,y) = 0.
    \end{align}
    Then there are $P\in L^1(X, \mu)$ and $Q\in L^1(Y, \nu)$ such that $f(x,y)=P(x)+Q(y)$ for $\gamma$-a.e.~$(x,y)\in X\times Y$.
\end{lemma}
\begin{proof}
    This is an application of the Hahn-Banach separation theorem. Define
    \[ U:=\{ P(x)+Q(y) \,|\, P\in L^1(X, \mu),\,Q\in L^1(Y, \nu) \} .  \]
    This is a subspace of $L^1(X\times Y, \gamma)$, since
    \begin{align*}
        &\int_{X\times Y} |P(x)+Q(y)|\dd\gamma(x,y) \\
        &\le \int_X |P(x)| \left(\int_Y u(x,y)\dd\nu(y)\right) \dd\mu(x)
        + \int_Y |Q(y)| \left(\int_X u(x,y)\dd\mu(x)\right) \dd\nu(y) \\
        &\qquad = \|P\|_{L^1(X,\mu)}+\|Q\|_{L^1(Y,\nu)}.
    \end{align*}
    Further, the subspace $U$ is closed since convergence of $f_n$ with $f_n(x,y)=P_n(x)+Q_n(y)$ in $L^1(X\times Y, \gamma)$ is equivalent to convergence of $P_n$ and $Q_n$ in the respective Banach spaces $L^1(X,\mu)$ and $L^1(Y,\nu)$. 
    
    Fix some $f_*\in L^1(X\times Y, \gamma)$ that is \emph{not} in $U$; the claim is proven if we can verify the existence of some $\theta_*\in L^\infty(X\times Y, \mu \otimes \nu)$ that satisfies \eqref{eq:app-adm}, but \emph{not} \eqref{eq:app-vanish} with $f=f_*$. Since $U$ is closed, $f_*$ has a positive distance to $U$, and the Hahn-Banach theorem thus guarantees the existence of a continuous linear functional $\varphi_*:L^1(X\times Y,\gamma)\to\R$ with $\varphi_*(U)=\{0\}$ and $\varphi_*(f_*)=1$. By $L^1$-$L^\infty$ duality, there exists some $\theta_*\in L^\infty(X\times Y, \mu\otimes \nu)$ such that
    \begin{align*}
        \varphi_*(f) = \int_{X\times Y} f\,\theta_*\dd\gamma \quad \text{for all $f\in L^1(X\times Y, \gamma)$.}
    \end{align*}
    Further, $\varphi_*(U)=\{0\}$ implies for $f(x,y):=a(x)+b(y)$ with arbitrary $a\in L^\infty(X, \mu)$ and $b\in L^\infty(Y, \nu)$ that 
    \begin{align*}
        0 = \varphi_*(f) 
        = \int_X a(x)\left(\int_Y\theta_*(x,y)u(x,y)\dd\nu(y)\right)\dd\mu(x)
        + \int_Y b(y)\left(\int_X\theta_*(x,y)u(x,y)\dd\mu(x)\right)\dd\nu(y).
    \end{align*}
    Hence $\theta_*$ satisfies the conditions \eqref{eq:app-adm}. On the other hand, one has
    \begin{align*}
        1 = \varphi_*(f_*) 
        = \int_{X\times Y} f_*\,\theta_*\dd\gamma,
    \end{align*}
    i.e., $f_*$ does \emph{not} satisfy \eqref{eq:app-vanish} for the choice $\theta=\theta_*$.
\end{proof}
\begin{remark}
    In case that $\gamma=\mu\otimes\nu$ is a product measure, i.e., $u\equiv1$, one does not need the abstract machinery of the Hahn-Banach theorem. Then, for any given $f_*\in L^1(X\times Y, \gamma)$ that is not of the form $f(x,y)=P(x)+Q(y)$, a suitable test function $\theta_*$ satisfying \eqref{eq:app-adm} but not \eqref{eq:app-vanish} for $f=f_*$ can be obtained very explicitly. The generalization of that construction to $\gamma$ that are not of product form is unclear. 
\end{remark}

\subsection*{Acknowledgements}
CC acknowledges partial support from the French National Research Agency through LabEx CEMPI (ANR-11-LABX-0007) and project MATHSOUT of the PEPR Mathematics in Interaction (ANR-23-EXMA-0010).
DM's research was supported by the DFG Collaborative Research Center TRR 109, ``Discretization in Geometry and Dynamics''.
IM and BS were supported by the Emmy Noether programme of the DFG. 

\bibliographystyle{plain}
\bibliography{main}
	
\end{document}